\documentclass[12pt]{article}

\usepackage{mathtools}
\usepackage{amssymb}
\usepackage{amsthm}
\usepackage{dsfont}
\allowdisplaybreaks

\DeclareSymbolFont{bbold}{U}{bbold}{m}{n}
\DeclareSymbolFontAlphabet{\mathbbold}{bbold}

\newcommand{\Ni}{\mathbb{N}}

\newcommand{\Ri}{\mathbb{R}}
\newcommand{\C}{\mathbb{C}}
\newcommand{\Ci}{\mathbb{C}}

\newcommand{\loc}{\mathrm{loc}}

\DeclareMathOperator{\BD}{BD}
\DeclareMathOperator{\Tr}{Tr}
\DeclareMathOperator{\dom}{dom}
\DeclareMathOperator{\ran}{ran}
\DeclareMathOperator{\rge}{ran}
\DeclareMathOperator{\kar}{ker}
\DeclareMathOperator{\mul}{mul}
\DeclareMathOperator{\tr}{Tr}
\DeclareMathOperator{\dive}{div}
\newcommand{\divv}{{\mathop{\rm div}}}
\DeclareMathOperator{\grad}{grad}
\newcommand{\bull}[1]{\dot{#1}}

\DeclareMathAccent{\Circ}{\mathalpha}{operators}{"17}
\newcommand{\interior}[1]{\Circ{#1}}

\newcommand{\RRe}{\operatorname{Re}}
\newcommand{\IIm}{\operatorname{Im}}
\newcommand{\curl}{\operatorname{curl}}

\DeclareMathOperator{\DtN}{\Lambda}

\let\leq\leqslant

\let\geq\geqslant

\makeatletter
\def\@row#1,{#1\@ifnextchar;{\@gobble}{&\@row}}
\def\@matrix{%
    \expandafter\@row\my@arg,;%
    \@ifnextchar({\\ \get@in@paren{\@matrix}}{\after@matrix}%
    }
\def\matrixtype#1#2#3{%
    \ifmmode\def\after@matrix{\end{#2}\right#3}%
    \else\def\after@matrix{\end{#2}\right#3$}$\fi
    \left#1\begin{#2}\get@in@paren{\@matrix}%
    }
\def\@column#1,{#1\@ifnextchar;{\@gobble}{\\ \@column}}
\newcommand\vect{}
\def\svect(#1){\left(\begin{smallmatrix}\@column#1,;\end{smallmatrix}\right)}
\def\vect{\get@in@paren{\@vect}}
\def\@vect{\left(\begin{matrix}\expandafter\@column\my@arg,;\end{matrix}\right)}
\def\get@in@paren#1({\def\my@arg{}\def\my@rest{}\def\after@get{#1}\get@arg}
\let\e@a\expandafter
\def\get@arg#1){\e@a\kl@test\my@rest#1(;}
\def\kl@test#1(#2;{\e@a\def\e@a\my@arg\e@a{\my@arg#1}%
                   \ifx:#2:\let\my@exec\after@get
                   \else\let\my@exec\get@arg
                        \e@a\def\e@a\my@arg\e@a{\my@arg(}%
                        \def@rest#2;%
                   \fi\my@exec}
\def\def@rest#1(;{\def\my@rest{#1\kl@zu}}
\def\kl@zu{)}

\makeatletter
\newcommand\MyPairedDelimiter{%
  \@ifstar{\My@Paired@Delimiter{{}}}
          {\My@Paired@Delimiter{}}%
}
\newcommand\My@Paired@Delimiter[4]{%
  \newcommand#2{%
    \@ifstar{\start@PD{#1}{\delimitershortfall=-1sp}{#3}{#4}}
            {\start@PD{#1}{}{#3}{#4}}%
  }%
}
\newcommand\start@PD[5]{%
  #1\mathopen{\mathpalette\put@delim@helper{\put@delim{#2}{#3}{.}{#5}}}%
  #5%
  \mathclose{\mathpalette\put@delim@helper{\put@delim{#2}{.}{#4}{#5}}}%
}
\newcommand\put@delim@helper[2]{%
  \hbox{$\m@th\nulldelimiterspace=0pt #2#1$}%
}
\newcommand\put@delim[5]{%
  \setbox\z@\hbox{$\m@th#5{#4}$}%
  \setbox\tw@\null
  \ht\tw@\ht\z@ \dp\tw@\dp\z@
  #1#5%
  \left#2\box\tw@\right#3%
}

\makeatother
\MyPairedDelimiter*{\abs}{\lvert}{\rvert}
\MyPairedDelimiter*{\norm}{\lVert}{\rVert}
\MyPairedDelimiter{\set}{\{}{\}}

\theoremstyle{plain} 
\newtheorem{theorem}{Theorem}[section]
\newtheorem{thm}[theorem]{Theorem}

\newtheorem{cor}[theorem]{Corollary}
\newtheorem{lemma}[theorem]{Lemma}

\newtheorem{prop}[theorem]{Proposition}

\theoremstyle{definition}
\newtheorem{example}[theorem]{Example}
\newtheorem{exam}[theorem]{Example}

\newtheorem{definition}[theorem]{Definition}
\newtheorem{remark}[theorem]{Remark}

\newcounter{teller}

\newenvironment{tabel}{\begin{list}%
{\rm  (\alph{teller})\hfill}{\usecounter{teller} \leftmargin=1.1cm
\labelwidth=1.1cm \labelsep=0cm \parsep=0cm}
                      }{\end{list}}

\newcounter{tellerr}

\newenvironment{tabeleq}{\begin{list}%
{\rm  (\roman{tellerr})\hfill}{\usecounter{tellerr} \leftmargin=1.1cm
\labelwidth=1.1cm \labelsep=0cm \parsep=0cm}
                         }{\end{list}}

\DeclareMathAlphabet\gothic{U}{euf}{m}{n}

\newcommand{\gotb}{\gothic{b}}

\makeatletter
\def\eqnarray{\stepcounter{equation}\let\@currentlabel=\theequation
\global\@eqnswtrue
\tabskip\@centering\let\\=\@eqncr
$$\halign to \displaywidth\bgroup\hfil\global\@eqcnt\z@
  $\displaystyle\tabskip\z@{##}$&\global\@eqcnt\@ne
  \hfil$\displaystyle{{}##{}}$\hfil
  &\global\@eqcnt\tw@ $\displaystyle{##}$\hfil
  \tabskip\@centering&\llap{##}\tabskip\z@\cr}

\def\endeqnarray{\@@eqncr\egroup
      \global\advance\c@equation\m@ne$$\global\@ignoretrue}

\def\@yeqncr{\@ifnextchar [{\@xeqncr}{\@xeqncr[5pt]}}
\makeatother

\newcommand{\cl}{{\cal L}}

\begin{document}
\bibliographystyle{tom}

\medmuskip=4mu plus 2mu minus 3mu
\thickmuskip=5mu plus 3mu minus 1mu
\belowdisplayshortskip=9pt plus 3pt minus 5pt

\title{The Dirichlet-to-Neumann operator for divergence form problems}

\author{A.F.M. ter Elst, G. Gordon and M. Waurick}

\date{\today}

\maketitle

\begin{abstract} 
We present a way of defining the Dirichlet-to-Neumann operator on
general Hilbert spaces using a pair of operators for which each one's 
adjoint is formally the negative of the other.
In particular, we define an abstract
analogue of trace spaces and are able to give meaning to the
Dirichlet-to-Neumann operator of divergence form operators perturbed
by a bounded potential in cases where the boundary of the underlying
domain does not allow for a well-defined trace.
Moreover, a
representation of the Dirichlet-to-Neumann operator as a first-order
system of partial differential operators is provided.
Using this
representation, we address  convergence of the Dirichlet-to-Neumann operators
in the case that the appropriate
reciprocals of the leading coefficients converge in the weak operator topology.
We also provide some extensions to the
case where the bounded potential is not coercive and consider resolvent 
convergence.
\end{abstract}

Keywords: Dirichlet-to-Neumann operator, resolvent convergence,
continuous dependence on the coefficients.

MSC 2010: 35F45, 46E35, 47A07.

\section{Introduction} \label{Sbdgd1}

In the theory of elliptic partial differential operators, the
Dirichlet-to-Neumann operator is a central object of study.
In recent
years it attracted a lot of attention and triggered profound research
in many directions.
In particular, we mention applications of the form
method, relations to the extension theory of symmetric operators as
well as the intimate connection to the Calder\'on problem, see, for
instance, the references in \cite{BeE1}.

The Dirichlet-to-Neumann operator relates Dirichlet
boundary data to the corresponding Neumann boundary data of solutions
to a partial differential equation.
As an introduction, we provide a
definition for the Dirichlet-to-Neumann operator in its
arguably simplest form.

Let $\Omega\subset \mathbb{R}^d$ be a bounded domain with smooth
boundary $\Gamma = \partial\Omega$ and where $d \geq 2$.
Note that in this case, the
trace map $\Tr$ from $H^1(\Omega)$ into $H^{1/2}(\Gamma)$ is
a well-defined, surjective and continuous operator.
Let $\varphi \in H^{1/2}(\Gamma)$ and let $u\in H^1(\Omega)$ be the solution of the
boundary value problem
\[
-\Delta u = 0 \mbox{ weakly on } \Omega
\quad \mbox{and} \quad \Tr u = \varphi.
\]
The Dirichlet-to-Neumann operator $\DtN$ assigns to $\varphi$ the normal
derivative of $u$, that is, 
$\DtN \varphi  = \partial_\nu u\in H^{-1/2}(\Gamma)$.

We can also consider the part of $\DtN$ in $L_2(\Gamma)$.
 If we call this restriction
$\DtN_{L_2(\Gamma)}$, then $\DtN_{L_2(\Gamma)}$ is an
unbounded operator in $L_2(\Gamma)$ such that for all $\varphi,\psi \in L_2(\Gamma)$
it follows that  $\varphi \in \dom(\DtN_{L_2(\Gamma)})$ and $\Lambda_{L_2(\Gamma)} \varphi = \psi$
if and only if
there exists a $u \in H^1(\Omega)$ such that $-\Delta u = 0$ weakly on $\Omega$,
$\Tr u = \varphi$ and $\psi = \partial_\nu u$.
A problem with the above descriptions is that they only make sense if the
boundary of $\Omega$ is sufficiently smooth.
We may also refer
to \cite{AE3} for a variant of the Dirichlet-to-Neumann operator for
domains with a rough boundary that has finite $(d-1)$-dimensional
Hausdorff measure.
If, however, $\Omega$ has for example a fractal
boundary with infinite $(d-1)$-dimensional Hausdorff measure, then 
in \cite{AE3} there is no notion of the Dirichlet-to-Neumann
operator at hand simply because there is no appropriate notion of a trace.
Using the concepts developed in \cite{PTW2} (with
extensions in \cite{PTW1} and \cite{Trostorff}), we are able to provide a 
substitute for the space $H^{1/2}(\Gamma)$.
We note here that this `trace-free' concept has proven to be useful
for dealing with boundary value problems on domains with rough boundary,
see \cite{PSTW}.

The substitute for the space $H^{1/2}(\Gamma)$ is a variant of $1$-harmonic
functions in $\Omega$.
This removes the need for function evaluation at the boundary.
For the definition of this substitute of $H^{1/2}(\Gamma)$, the only
concept that we use, if we relate our findings to the
Laplacian, is that the matrix 
$\bigl(\begin{smallmatrix} 0 & \divv \\ \grad & 0\end{smallmatrix}\bigr)$
is skew-symmetric on the space of infinite differentiable functions
with compact support, see Example~\ref{ex:classical}.
Thus, without further effort, our results directly apply to
similar problems involving the equations of linearized elasticity or
the full 3-dimensional system of static Maxwell's equations.
More
generally, our methods apply to the covariant derivative defined on
suitable $L_2$-tensor fields and a formal skew-adjoint.

As our central object of study, we shall deviate from the classical
elliptic partial differential operator $-\Delta$ discussed above and 
treat abstract divergence form operators of the form
\begin{equation}\label{eq:varcoef}
-DaG + m,
\end{equation}
where $a$ and $m$ are bounded coercive operators (called coefficients) 
and $D$ and $G$ are densely
defined, closed, unbounded operators in Hilbert spaces $H_1$ and $H_0$
with the property $-D^*\subset G$, like $\divv$ and $\grad$.

If $\dom(G)$, endowed with the graph norm, embeds compactly into $H_0$,
we will also address the concept of
continuous dependence of the Dirichlet-to-Neumann operator associated
with \eqref{eq:varcoef} on the bounded coefficients $a$ and $m$ under
the weak operator topology.
This result has applications in
homogenization problems, see \cite{Tartar} and \cite{Waurick1} Section~5.5.
Moreover, it complements the study of continuous dependence of
the Dirichlet-to-Neumann operator on its coefficients in \cite{AEKS},
where the authors focus on possible non-coercive cases and 
convergence of the principal coefficients in $L_\infty(\Omega)$.
In order to prove convergence results, 
we derive a reformulation of the Dirichlet-to-Neumann
operator as a system of two first-order partial differential
equations, similar to \cite{AKM}.

In the present work we also consider removing the coercivity condition on~$m$.
That is
to say, we define the abstract analogue of the Dirichlet-to-Neumann
\emph{graph} with $m$ being possibly not coercive.
We note here
that these results are the abstract counterpart of results developed
in \cite{BeE1} and \cite{AEKS}.
In the case that the potentials $m$ are not coercive
we consider resolvent convergence for Dirichlet-to-Neumann {\it operators}.
Under different assumptions convergence of Neumann-to-Dirichlet operators
was obtained in \cite{Rondi}.

We mention here that a possible non-linear variant of the
Dirichlet-to-Neumann operator, where the coercive operator $a$ is replaced
by a (strictly) maximal monotone relation, can be discussed using the
results of \cite{TrostorffWaurick}.
This however is beyond the scope of the
present manuscript and will be addressed in future work.

We briefly comment on the organization of the paper.
In Section~\ref{Sbdgd2}, we provide the basic functional analytic setting and
recall some notions and results of \cite{PTW2}, \cite{PTW1} and \cite{Trostorff}.
We then state the definition of the Dirichlet-to-Neumann operator in
the abstract setting discussed above.
We also provide an extensive
example that justifies this abstraction by relating it to the
classical formulation of the Dirichlet-to-Neumann operator.
In Section~\ref{Sbdgd3} we give a representation formula for the
Dirichlet-to-Neumann operator as a first-order system and show that
this operator is m-sectorial, provided both $m$ and~$a$ are coercive.
For this we use a representation result for operators given via forms,
see \cite{AE2}.
In Section~\ref{Sbdgd4} we prove resolvent convergence of the 
Dirichlet-to-Neumann operators when the coefficients 
converge in an appropriate weak operator topology.
Under some additional hypotheses we also obtain in Theorem~\ref{tbdgd403}
uniform convergence even though the
coefficients converge in the weak operator topology only.
In Section~\ref{Sbdgd5} we consider the non-coercive case and
discuss the domain and multi-valued parts of the
Dirichlet-to-Neumann \emph{graph} when $m$ is merely
assumed to be a bounded operator, that is not necessarily coercive.
Moreover, we also prove a convergence theorem for the non-coercive case
in Section~\ref{Sbdgd6}.
We conclude with two more examples in Section~\ref{Sbdgd7}.

\section{The Dirichlet-to-Neumann operator and boundary spaces} \label{Sbdgd2}

We start with a description of boundary data spaces as in \cite{PTW2} Subsection~5.2.
Throughout this paper fix Hilbert spaces $H_0$ and $H_1$.
Further, let $G$ be an operator in $H_0$ with values in $H_1$ and 
let $D$ be an operator in $H_1$ with values in $H_0$.
We assume throughout that both $G$ and $D$ are densely defined and closed, and that 
$-G^*\subset D$.
We define $\interior{D} = -G^*$ and $\interior{G} = -D^*$.

Note that 
\[
(\interior{G} u, q)_{H_1} 
= - (u, \interior{D} q)_{H_0}
\]
for all $u \in \dom(\interior{G})$ and $q \in \dom(\interior{D})$.
Equivalently, the matrix 
\[
\begin{pmatrix} 0 & \interior{D} \\ \interior{G} & 0\end{pmatrix}
\]
with dense domain $\dom(\interior{G}) \times \dom(\interior{D})$ is 
skew-symmetric in $H_0 \times H_1$.

\begin{remark} \label{rem:intG}
  Note that $\interior{G}= -D^* \subset -(-G^*)^*=\overline{G}=G$. 
So one can simultaneously swap $H_0$ with $H_1$ 
and $D$ with $G$.
\end{remark}

\begin{exam} \label{xbdgs201}
All examples in this paper are of the following type.
Let $H_0$ and $H_1$ be Hilbert spaces.
Consider dense subspaces $\dom(\widehat G) \subset H_0$ and 
$\dom(\widehat D) \subset H_1$.
Let $\widehat G \colon \dom(\widehat G) \to H_1$ and 
$\widehat D \colon \dom(\widehat D) \to H_0$ be two operators
such that 
\begin{equation}
(\widehat G u, q)_{H_1} 
= - (u, \widehat D q)_{H_0}
\label{exbdgs201;1}
\end{equation}
for all $u \in \dom(\widehat G)$ and $q \in \dom(\widehat D)$.
Equivalently, the matrix 
\[
\begin{pmatrix} 0 & \widehat D \\ \widehat G & 0\end{pmatrix}
\]
with dense domain $\dom(\widehat G) \times \dom(\widehat D)$ is 
skew-symmetric in $H_0 \times H_1$.

Then $\widehat G \subset - (\widehat D)^*$ and 
$\widehat D \subset - (\widehat G)^*$, so both $\widehat G$ and $\widehat D$
are closable.
Let $\interior{G}$ and $\interior{D}$ denote the closures.
Define $G = - (\interior{D})^*$ and $D = - (\interior{G})^*$.
Since $\interior{D}$ and $\interior{G}$ are closed, therefore closable, it follows that 
$G$ and $D$ are densely defined.
Obviously both $G$ and $D$ are closed.
Next $G^* = - (\interior{D})^{**} = - \interior{D}$ since $\interior{D}$ is closed
and similarly $D^* = - \interior{G}$.
Also $\interior{D} \subset - (\widehat G)^* = - (\interior{G})^*$, 
so $\interior{G} \subset - (\interior{D})^* = G$.
Similarly $\interior{D} \subset D$.
Then $G^* = - \interior{D} \subset - D$ as required.
\end{exam}

The classical example for this paper is as follows.
Note that we do not assume any condition on the boundary of~$\Omega$.

\begin{example} \label{ex:classical}
 Let $\Omega\subset \Ri^d$ be open.
Define $\widehat G \colon C_c^\infty(\Omega) \to L_2(\Omega)^d$ and 
$\widehat D \colon C_c^\infty(\Omega)^d \to L_2(\Omega)$ by
\[
\widehat G u = (\partial_1 u,\ldots,\partial_d u)
\quad \mbox{and} \quad 
\widehat D q = \sum_{k=1}^d \partial_k q_k
.  \]
Define $H_0 = L_2(\Omega)$ and $H_1 = L_2(\Omega)^d$.
Then (\ref{exbdgs201;1}) in Example~\ref{xbdgs201}
follows from integration by parts.
The associated operators are denoted by $G = \grad$, $\interior{G} = \interior{\grad}$,
$D = \divv$ and $\interior{D} = \interior{\divv}$.
It is not hard to show that $\dom(\interior{\grad}) = H^1_0(\Omega)$,
$\dom(\grad) = H^1(\Omega)$ and 
$\dom(\divv) = H_{\dive}(\Omega) = \{ q\in L_2(\Omega)^d : \dive q\in L_2(\Omega)\}$. 
\end{example}

We next define an (abstract) variant of the trace spaces $H^{1/2}(\Gamma)$ and 
$H^{-1/2}(\Gamma)$.
Throughout this paper we provide the domain of an operator with the graph norm.
Define
\[
  \BD(G) = \dom(\interior{G})^{\bot_{\dom(G)}} 
\quad \mbox{and} \quad
  \BD(D) = \dom(\interior{D})^{\bot_{\dom(D)}}.
\]
We provide $\BD(G)$ and $\BD(D)$ with the induced inner products of 
$\dom(G)$ and $\dom(D)$.
We denote by $\pi_{\BD(G)}$ and $\pi_{\BD(D)}$ the corresponding projections onto 
$\BD(G)$ and $\BD(D)$, respectively.

\begin{example} \label{xbdgd210}
Let $\Omega$, $G$ and $D$ be as in Example~\ref{ex:classical}.
Then $\BD(G) = \{ u \in H^1(\Omega) : \Delta u = u \mbox{ weakly on } \Omega \} $.
Indeed, let $u \in \BD(G)$. 
Then $u \in H^1(\Omega)$ and 
$0 = (u,v)_{\dom(G)} = (u,v)_{L_2(\Omega)} + (\grad u,\grad v)_{L_2(\Omega)}$
for all $v \in \dom(\interior{G}) = H^1_0(\Omega)$.
So $\Delta u = u$ weakly on $\Omega$.
The converse inclusion is similar.
\end{example}

\begin{lemma} \label{l:bd} 
$\BD(G)=\kar(I-DG)$ and $\BD(D)=\kar(I-GD)$.
\end{lemma}
\begin{proof}
By Remark~\ref{rem:intG} it suffices to prove the first equality.
Let $u\in \BD(G)$.
Then 
\[
(u,v)_{H_0}+(Gu,\interior{G}v)_{H_1}
= (u,v)_{\dom(G)}
=0
\]
for all $v \in \dom(\interior{G})$.
So $Gu \in \dom((\interior{G})^*) = \dom(D)$ and 
$D G u = - (\interior{G})^* G u = u$.
Therefore $u \in \kar(I-DG)$.
The converse follows similarly.
\end{proof}

\begin{cor} \label{cbdgd201}
If $u \in \BD(G)$, then $Gu \in \BD(D)$.
If $q \in \BD(D)$, then $Dq \in \BD(G)$.
\end{cor}
\begin{proof}
Let $u \in \BD(G)$. 
Then $u \in \dom(DG)$ and $DGu = u \in \dom(DG)$.
Therefore $u \in \dom(GDG)$ and 
$(I - GD) Gu = G(I - DG) u = 0$.
So $Gu \in \kar(I-GD) = \BD(D)$ by Lemma~\ref{l:bd}.
The other statement follows similarly.
\end{proof}

Define $\bull{G} \colon \BD(G)\to \BD(D)$ and $\bull{D}\colon \BD(D)\to \BD(G)$ by
\[
\bull{G} u = Gu
\quad \mbox{and} \quad 
\bull{D} q = Dq
.  \]

\begin{lemma} \label{l:Gdot} 
The operators $\bull{G}$ and $\bull{D}$ are unitary.
Moreover, $(\bull{G})^*=\bull{D}$.
\end{lemma}
\begin{proof}
See \cite[Theorem 5.2]{PTW2}.
For the convenience of the reader we include the proof.
Clearly $\bull{D}\bull{G}=I_{\BD(G)}$ and $\bull{G}\bull{D}=I_{\BD(D)}$ by 
Lemma~\ref{l:bd}.
Moreover, 
\[
(\bull{G}u,q)_{\BD(D)}
= (\bull{G}u,q)_{H_1} + (\bull{D}\bull{G}u,\bull{D}q)_{H_0} 
= (\bull{G}u,\bull{G}\bull{D}q)_{H_1} + (u,\bull{D}q)_{H_0}
= (u,\bull{D}q)_{\BD(G)}
\]
for all $u\in \BD(G)$ and $q\in \BD(D)$,
from which the lemma follows.
\end{proof}

In the situation of Example~\ref{ex:classical}
the space $\BD(G)$ models the boundary data of an $H^1(\Omega)$-function
if $\Omega$ is a bounded Lipschitz domain, as shown 
in \cite[Corollary 4.4]{Trostorff}.
Indeed, let $\Gamma = \partial \Omega$.
Since $\Tr \colon H^1(\Omega) \to H^{1/2}(\Gamma)$ is 
continuous, surjective and $\ker \Tr = H^1_0(\Omega) = \dom(\interior{\grad})$,
it follows that 
\begin{equation}
\Tr|_{\BD(G)} \colon \BD(G) \to H^{1/2}(\Gamma)
\label{eSbdgd2;1}
\end{equation}
is bijective
and hence a topological isomorphism.

We next consider the space $\BD(D)$.
Denote by $\BD(G)'$ the space of all antilinear continuous 
maps from $\BD(G)$ into~$\Ci$.
There is a natural unitary map from $\BD(D)$ onto $\BD(G)'$.

\begin{prop} \label{pbdgd202}
Define $\Phi \colon \BD(D) \to \BD(G)'$ by 
\[
\Big( \Phi(q) \Big) (u)
= (D q, u)_{H_0} + (q, G u)_{H_1}
.  \]
Then $\Phi$ is unitary.
\end{prop}
\begin{proof}
Let $q \in \BD(D)$ and $u \in \BD(G)$.
Then 
\begin{eqnarray}
\Big( \Phi(q) \Big) (u)
& = & (D q, u)_{H_0} + (q, G u)_{H_1}  \nonumber  \\
& = & (q, G u)_{H_1} + (D q, D G u)_{H_0}
= (q, G u)_{\dom(D)}
= (q, \bull{G} u)_{\BD(D)}
.  
\label{etbdgd202;1}
\end{eqnarray}
Then the proposition follows from Lemma~\ref{l:Gdot} and the Riesz representation 
theorem.
\end{proof}

For clarity and contrast we include the proof of the next proposition.
We provide $\Tr H^1(\Omega)$ with the quotient norm.

\begin{prop} \label{pbdgd202.5}
Let $\Omega \subset \Ri^d$ be open, bounded with Lipschitz boundary.
Then one has the following.
\begin{tabel}
\item \label{pbdgd202.5-1}
For all $q \in H_\divv(\Omega)$ there exists a unique $Q \in (\Tr H^1(\Omega))'$
such that 
\begin{equation}
\langle Q, \Tr u \rangle_{(\Tr H^1(\Omega))' \times \Tr H^1(\Omega)}
= \int_\Omega (\divv q) \overline u + \int_\Omega q \cdot \overline{\nabla u}
\label{epbdgd202.5;1}
\end{equation}
for all $u \in H^1(\Omega)$.
\item \label{pbdgd202.5-2}
If $q \in \dom(\interior{\divv})$, then $Q = 0$, where $Q$ is as in (\ref{epbdgd202.5;1}).
\item \label{pbdgd202.5-3}
If $q \in H^1(\Omega)^d$, then $Q = \nu \cdot \Tr q$,
where $\nu$ is the outward normal vector on the boundary $\Gamma$ of $\Omega$
and $Q$ is as in (\ref{epbdgd202.5;1}).
\end{tabel}
\end{prop}
\begin{proof}
`\ref{pbdgd202.5-1}'.
Define $F \colon H^1(\Omega) \to \Ci$ by
\[
F(u) = \int_\Omega (\divv q) \overline u + \int_\Omega q \cdot \overline{\nabla u}
.  \]
Then $F \in H^1(\Omega)'$.
Moreover, if $u \in H^1_0(\Omega)$, then $F(u) = 0$.
Hence there exists a unique continuous antilinear map 
$\widetilde F \colon \Tr H^1(\Omega) \to \Ci$ such that 
$\widetilde F(\Tr u) = F(u)$ for all $u \in H^1(\Omega)$.
Then the first statement follows.

`\ref{pbdgd202.5-2}'.
We use the notation as in Example~\ref{ex:classical}.
Let $q \in \dom(\interior{D})$.
Since $\interior{D} = - G^*$ one deduces that 
$F(u) = \int_\Omega (\divv q) \overline u + \int_\Omega q \cdot \overline{\nabla u}
= (\interior{D} q, u)_{H_1} + (q,G u)_{H_0} = 0$
for all $u \in \dom(G)$.
So $Q = 0$, because $\dom(G)$ is dense in $H^1(\Omega)$.

`\ref{pbdgd202.5-3}'.
Suppose that $q \in H^1(\Omega)^d$.
Let $u \in H^1(\Omega)$.
Then $\overline u q \in W^{1,1}(\Omega)^d$ and the divergence theorem gives
\[
\int_\Omega (\divv q) \overline u + \int_\Omega q \cdot \overline{\nabla u}
= \int_\Omega \divv(\overline u q)
= \int_\Gamma \nu \cdot \Tr(\overline u q)
= \int_\Gamma (\nu \cdot \Tr q) \Tr \overline u
.  \]
So $Q = \nu \cdot \Tr q$.
\end{proof}

If $q \in H_\divv(\Omega)$ and $Q$ is as in Proposition~\ref{pbdgd202.5}, 
then we define $(\nu q) = Q$.
So $(\nu q) = \nu \cdot \Tr q$ if $q \in H^1(\Omega)^d$.

\begin{exam} \label{xbdgd202.6}
Let $\Omega$ be a bounded Lipschitz domain with boundary $\Gamma$.
Let $G$ and $D$ be as in Example~\ref{ex:classical}.
Let $\Phi$ be as in Proposition~\ref{pbdgd202}.
Then 
\[
(\Phi(q))(u) = \langle (\nu q), \Tr u \rangle_{(\Tr H^1(\Omega))' \times \Tr H^1(\Omega)}
\]
for all $q \in \BD(D)$ and $u \in \BD(G)$.

It follows from (\ref{eSbdgd2;1}) and Proposition~\ref{pbdgd202}
that the spaces $\BD(D)$ and $H^{-1/2}(\Gamma)$
are isomorphic.
Hence $\bull{G}$ is a variant of the Dirichlet-to-Neumann operator.
\end{exam}

Next we introduce the (variable) coefficients for our abstract Dirichlet-to-Neumann operator.
Recall that a bounded operator $M$ in a Hilbert space $H$ is called {\bf coercive}
if there exists a $\mu > 0$ such that $\RRe M \geq \mu I$, where $\RRe M = \frac{1}{2} (M + M^*)$.
That is $M$ is coercive if and only if there exists a $\mu > 0$ such that 
$\RRe (Mx,x) \geq \mu \|x\|_H^2$ for all $x \in H$.

As for the classical Dirichlet-to-Neumann operator, we first show that the 
Dirichlet problem has a unique solution.

\begin{prop} \label{pbdgd203}
Let $a\in \cl(H_1)$ and $m \in \cl(H_0)$ be coercive.
Let $u_0\in \BD(G)$.
Then there exists a unique $u\in \dom(DaG)$ such that
$mu -DaGu=0$ and $u-u_0\in \dom(\interior{G})$.
\end{prop}

For the proof of the proposition we need several auxiliary results.

\begin{lemma}\label{l:wp} 
Let $H$ be a Hilbert space, $M\in \cl(H)$ and $A$ a skew-adjoint operator in~$H$.
Let $\lambda > 0$ and assume that 
$\RRe (M x, x)_H \geq \lambda \|x\|_H^2$ for all $x \in H$.
Then the operator $M+A$ is invertible.
Moreover, the operator $(M+A)^{-1}$ is bounded from $H$ into $\dom(A)$
and $\|(M+A)^{-1}\|_{H \to \dom(A)} \leq \frac{1 + \lambda + \|M\|}{\lambda}$.
\end{lemma}
\begin{proof}
If $x\in \dom(A)$, then
$\RRe((M+A)x,x)_H=\RRe((Mx,x)_H\geq \lambda \|x\|^2_H$.
Hence $M+A$ is one-to-one, its range is closed and 
$M+A$ is continuously invertible on its range.
Since $\RRe(Mx,x)_H = \RRe(M^*x,x)_H$ for all $x \in H$, 
we obtain similarly that $(M+A)^*=M^*-A$ is one-to-one.
Therefore $M+A$ is onto.
So $M + A$ is invertible and $\|(M+A)^{-1}\|_{H \to H} \leq \frac{1}{\lambda}$.

Since $A (M+A)^{-1} = I - M (M+A)^{-1}$, the operator $A (M+A)^{-1}$ 
is bounded from $H$ into $H$ and the estimate follows.
\end{proof}

Next we consider matrix operators.

\begin{lemma}\label{l:1storder}
Let $a\in \cl(H_1)$ and $m \in \cl(H_0)$ be coercive.
\begin{tabel}
\item \label{l:1storder-1}
The operators $\begin{pmatrix} m & -\interior{D} \\ -G & a^{-1}\end{pmatrix}$ 
and $\begin{pmatrix} m & -{D} \\ -\interior{G} & a^{-1}\end{pmatrix}$
in $H_0\times H_1$ are invertible.  
\item \label{l:1storder-2}
The operator $\begin{pmatrix} m & -\interior{D} \\ -G & a^{-1}\end{pmatrix}^{-1}$ 
is bounded from $H_0\times H_1$ into $\dom(G) \times \dom(\interior{D})$.
\item \label{l:1storder-3}
The operator $\begin{pmatrix} m & -{D} \\ -\interior{G} & a^{-1}\end{pmatrix}^{-1}$
is bounded from $H_0\times H_1$ into $\dom(\interior{G}) \times \dom(D)$.
\end{tabel}
\end{lemma}
\begin{proof}
Let $H=H_0\times H_1$, $M=\begin{pmatrix} m & 0 \\ 0 & a^{-1}\end{pmatrix}$
and $A=\begin{pmatrix} 0 & -\interior{D} \\ -G & 0\end{pmatrix}$
with $\dom(A) = \dom(G) \times \dom(\interior{D})$.
Since $-\interior{D}^*=G$ and $- G^* = \interior{D}$, the operator $A$ is 
skew-adjoint.
Also $\RRe a^{-1} \geq \|a\|^{-2} \RRe a$, so $M$ is coercive.
Therefore $M+A$ is invertible and the operator $(M + A)^{-1}$
is bounded from $H$ into $\dom(A)$ by Lemma~\ref{l:wp}.
This proves the first part of Statement~\ref{l:1storder-1} and Statement~\ref{l:1storder-2}

The remaining parts of the lemma follow similarly.
\end{proof}

\begin{lemma} \label{l:invdes} 
Let $a\in \cl(H_1)$ and $m \in \cl(H_0)$ be coercive.
Let $u\in \dom(G)$, $q\in \dom(D)$, $u_0\in \BD(G)$ and $q_0\in \BD(D)$. 
\begin{tabel}
\item  \label{l:invdes-1} 
The following conditions are equivalent.
\begin{tabeleq}
\item  \label{l:invdes-1-1} 
$Dq=mu$, $q=aGu$ and $u-u_0\in \dom(\interior{G})$.
\item  \label{l:invdes-1-2} 
$q=aGu$, $u-u_0\in \dom(\interior{G})$  and
\[
    (aGu,\interior{G}v)_{H_1}=-(mu,v)_{H_0}
\]
for all $v\in \dom(\interior{G})$.
\item  \label{l:invdes-1-3} 
$\begin{pmatrix} u-u_0\\ q\end{pmatrix} 
= \begin{pmatrix} m & -D \\ -\interior{G} & a^{-1}\end{pmatrix}^{-1}
  \begin{pmatrix} -mu_0 \\ Gu_0\end{pmatrix}$.
\end{tabeleq}
\item  \label{l:invdes-2} 
The following conditions are equivalent.
\begin{tabeleq}
\item  \label{l:invdes-2-1} 
$Dq=mu$, $q=aGu$ and $q-q_0\in \dom(\interior{D})$.
\item  \label{l:invdes-2-2} 
$\begin{pmatrix} u\\ q-q_0\end{pmatrix}
= \begin{pmatrix} m & -\interior{D} \\ -G & a^{-1}\end{pmatrix}^{-1}
  \begin{pmatrix}
   Dq_0 \\ -a^{-1}q_0
  \end{pmatrix}$.
\end{tabeleq}
\end{tabel}
\end{lemma}
\begin{proof}
`\ref{l:invdes-1}'. 
`\ref{l:invdes-1-1}$\Leftrightarrow$\ref{l:invdes-1-2}'.
This follows immediately from the equality $D = - (\interior{G})^*$.

`\ref{l:invdes-1-1}$\Leftrightarrow$\ref{l:invdes-1-3}'.
By a simple algebraic manipulation Condition~\ref{l:invdes-1-1} is equivalent to 
\[
u-u_0\in \dom(\interior{G})
\quad \mbox{and} \quad
\begin{pmatrix} m & -D \\ -G & a^{-1}\end{pmatrix} 
  \begin{pmatrix} u-u_0\\ q\end{pmatrix} 
= \begin{pmatrix} -mu_0 \\ Gu_0\end{pmatrix}
.  \]
By Lemma~\ref{l:1storder}\ref{l:1storder-1} this is equivalent to Condition~\ref{l:invdes-1-3}.

`\ref{l:invdes-2}'. 
The proof is similar.
\end{proof}

Now we are able to prove Proposition~\ref{pbdgd203}.

\begin{proof}[{\bf Proof of Proposition~\ref{pbdgd203}.}]
First we show existence.
Let $u\in \dom(G)$ and $q\in \dom(D)$ be such that
\[
\begin{pmatrix} u-u_0\\ q\end{pmatrix} 
= \begin{pmatrix} m & -D \\ -\interior{G} & a^{-1}\end{pmatrix}^{-1}
  \begin{pmatrix} -mu_0 \\ Gu_0 \end{pmatrix}.
\]
Then $u$ satisfies the desired properties by 
Lemma~\ref{l:invdes}\ref{l:invdes-1} \ref{l:invdes-1-3}$\Rightarrow$\ref{l:invdes-1-1}.

It remains to show uniqueness.
Let $\tilde u\in \dom(DaG)$ and suppose that
$m\tilde u -DaG\tilde u=0$ and $\tilde u-u_0\in \dom(\interior{G})$.
Set $\tilde q = aG\tilde u$.
Then it follows from 
Lemma~\ref{l:invdes}\ref{l:invdes-1} \ref{l:invdes-1-1}$\Rightarrow$\ref{l:invdes-1-3}
that 
\[
\begin{pmatrix} \tilde u-u_0\\ \tilde{q}\end{pmatrix} 
= \begin{pmatrix} m & -D \\ -\interior{G} & a^{-1}\end{pmatrix}^{-1}
  \begin{pmatrix} -mu_0 \\ Gu_0 \end{pmatrix},
\]
which implies that $u=\tilde u$.
\end{proof}

There is a similar version of Proposition~\ref{pbdgd203} for the Neumann problem.

\begin{prop} \label{pbdgd204}
Let $a\in \cl(H_1)$ and $m \in \cl(H_0)$ be coercive.
Let $q_0\in \BD(D)$.
Then there exists a unique $u\in \dom(DaG)$ such that
$mu -DaGu=0$ and $aGu - q_0\in \dom(\interior{D})$.
\end{prop}
\begin{proof}
This follows similarly to the proof of Proposition~\ref{pbdgd203}, 
but now use Lemma~\ref{l:invdes}\ref{l:invdes-2}
instead of Lemma~\ref{l:invdes}\ref{l:invdes-1}.
\end{proof}

At this stage we are able to define the Dirichlet-to-Neumann operator with 
variable coefficients as an operator 
acting from $\BD(G)$ (the abstract realization of $H^{1/2}(\Gamma)$) to 
$\BD(D)$ (the abstract realization of $H^{-1/2}(\Gamma)$).

\begin{definition} \label{dbdgd2-4.5}
Let $a\in \cl(H_1)$ and $m \in \cl(H_0)$ be coercive.
Define the operator
\[
   \DtN \colon \BD(G)\to \BD(D)
\]
as follows.
Let $u_0\in \BD(G)$.
By Proposition~\ref{pbdgd203} there exists a unique
$u\in \dom(DaG)$ such that 
$mu -DaGu=0$ and $u-u_0\in \dom(\interior{G})$.
Then we define $\Lambda u_0 = \pi_{\BD(D)}aGu$.
We call $\Lambda$ the {\bf Dirichlet-to-Neumann operator associated with $- DaG + m$}.
\end{definition}

So the graph of the operator $\Lambda$ is equal to 
\[
\{ (\pi_{\BD(G)} u, \pi_{\BD(D)} aGu) : 
      u \in \dom(DaG) \mbox{ and } mu -DaGu=0 \}
.  \]

\begin{theorem} \label{t:DtNwd} 
Let $a \in \cl(H_1)$ and $m \in \cl(H_0)$ be coercive.
Then the operator $\DtN$ associated with $- DaG + m$ is bounded and invertible.
Moreover, 
\[
   \DtN u_0 
= \begin{pmatrix} 0 & \pi_{\BD(D)}\end{pmatrix} 
  \begin{pmatrix} m & -D \\ -\interior{G} & a^{-1}\end{pmatrix}^{-1}
  \begin{pmatrix}-m\\ G\end{pmatrix} u_0
\] 
for all $u_0\in \BD(G)$ and
\[
   \DtN^{-1} q_0 
= \begin{pmatrix} \pi_{\BD(G)} & 0 \end{pmatrix} 
  \begin{pmatrix} m & -\interior{D} \\ -G & a^{-1}\end{pmatrix}^{-1}
  \begin{pmatrix} D\\ -a^{-1}\end{pmatrix} q_0
\] 
for all $q_0\in \BD(D)$.
\end{theorem}
\begin{proof} 
The expression for $\Lambda$ follows from Lemma~\ref{l:invdes}\ref{l:invdes-1},
arguing as in the proof of Proposition~\ref{pbdgd203}.
The boundedness of $\Lambda$ is then a consequence of
Lemma~\ref{l:1storder}\ref{l:1storder-3}.

The proof for $\Lambda^{-1}$ is similar, using Lemma~\ref{l:invdes}\ref{l:invdes-2},
Proposition~\ref{pbdgd204} and Lemma~\ref{l:1storder}\ref{l:1storder-2}.
\end{proof}

\section{An intermediate operator and m-sectoriality} \label{Sbdgd3}

In Proposition~\ref{pbdgd202} we showed that the space $\BD(D)$ is naturally isomorphic to 
$\BD(G)'$.
In this section we assume that there is a Hilbert space $H$ such that 
$\BD(G) \hookrightarrow H \hookrightarrow \BD(G)'$ is a Gelfand triple.
Then we study the part of the Dirichlet-to-Neumann operator in~$H$.
In the model example, Example~\ref{ex:classical}, one can take $H = L_2(\Gamma)$.

Throughout this section, we adopt the notation and assumptions 
as in the beginning of Section~\ref{Sbdgd2}.
In addition, let $H$ be a Hilbert space and $\kappa \in \cl(\BD(G),H)$.
We assume that $\kappa$ is one-to-one and has dense range.

\begin{example} \label{xbdgd205}
Let $\Omega$ be a bounded Lipschitz domain with boundary $\Gamma$.
Let $G$ and $D$ be as in Example~\ref{ex:classical}.
Let $\sigma \in (-\infty,\frac{1}{2}]$ and choose $H = H^\sigma(\Gamma)$.
Define $\kappa \colon \BD(G) \to H$ by $\kappa(u) = \tr u$.
Then $\kappa$ is one-to-one and has dense range.
Note that $\kappa$ is compact if and only if $\sigma < \frac{1}{2}$.

Now suppose that $\sigma = 0$, so $H = L_2(\Gamma)$.
Let $\psi \in L_2(\Gamma)$ and set $u = \kappa^* \psi$.
Then $u \in \BD(G)$, so $u \in H^1(\Omega)$ and 
$\Delta u = u$ weakly on $\Omega$ by Example~\ref{xbdgd210}.
If $v \in \BD(G)$, then 
\begin{eqnarray*}
\int_\Gamma \psi \overline{\Tr v}
& = & (\psi, \kappa(v))_{L_2(\Gamma)}
= (\kappa^* \psi, v)_{\BD(G)}
= (u, v)_{\BD(G)}  \\
& = & \int_\Omega u \overline v + \int_\Omega \nabla u \cdot \overline{\nabla v}
= \int_\Omega (\Delta u) \overline v + \int_\Omega \nabla u \cdot \overline{\nabla v}
.
\end{eqnarray*}
Alternatively, if $v \in H^1_0(\Omega) = \dom(\interior{G})$, then 
\[
\int_\Gamma \psi \overline{\Tr v} 
= 0 = \int_\Omega (\Delta u) \overline v + \int_\Omega \nabla u \cdot \overline{\nabla v}
.  \]
So by linearity
\[
\int_\Gamma \psi \overline{\Tr v} 
= \int_\Omega (\Delta u) \overline v + \int_\Omega \nabla u \cdot \overline{\nabla v}
\]
for all $v \in H^1(\Omega)$.
Hence $u$ has a weak normal derivative and $\partial_\nu u = \psi$.
\end{example}

We consider the Gelfand triple 
\[
\BD(G) 
\stackrel{\kappa}{\hookrightarrow}
H
\simeq
H'
\stackrel{\kappa'}{\hookrightarrow}
\BD(G)'
\]
with $H$ as pivot space.
Recall that $\BD(G)'$ is naturally isomorphic to $\BD(D)$ by Proposition~\ref{pbdgd202}.
We aim to describe the part of the Dirichlet-to-Neumann operator $\Lambda$ 
in~$H$.
We describe the image of $H$ in $\BD(D)$ under the above maps
$H \simeq H' \stackrel{\kappa'}{\hookrightarrow} \BD(G)' \simeq \BD(D)$.

\begin{lemma} \label{lbdgd206}
Let $\Phi \colon \BD(D) \to \BD(G)'$ be as in Proposition~\ref{pbdgd202}.
Define $F \colon H \to H'$ by $(F \varphi)(\psi) = (\varphi,\psi)_H$.
Then $\Phi^{-1} \circ \kappa' \circ F = G \circ \kappa^*$.
\end{lemma}
\begin{proof}
Let $\varphi \in H$ and write $q = (\Phi^{-1} \circ \kappa' \circ F)(\varphi)$.
Let $u \in \BD(G)$.
Then it follows from Lemma~\ref{l:Gdot} and (\ref{etbdgd202;1}) that 
\begin{eqnarray*}
(\bull{D} q, u)_{\BD(G)}
& = & (q,\bull{G} u)_{\BD(D)}  \\
& = & (\Phi(q))(u)
= ((\kappa' \circ F) \varphi)(u)
= (\varphi, \kappa(u))_H
= (\kappa^* \varphi, u)_{\BD(G)}
.
\end{eqnarray*}
So $\bull{D} q = \kappa^* \varphi$ and $q = \bull{G} \bull{D} q = \bull{G} \kappa^* \varphi$.
\end{proof}

Now we are able to define the part of the Dirichlet-to-Neumann operator in~$H$.

\begin{definition} \label{dbdgd301}
Let $a \in \cl(H_1)$ and $m \in \cl(H_0)$ be coercive.
Define the operator $\Lambda_H$ in $H$ as follows. 
Let $\varphi,\psi \in H$.
Then we say that $\varphi \in \dom(\Lambda_H)$ and $\Lambda_H \varphi = \psi$
if there exists a $u_0 \in \BD(G)$ such that $\kappa(u_0) = \varphi$
and $\Lambda u_0 = (G \circ \kappa^*)(\psi)$, where 
$\Lambda$ is the Dirichlet-to-Neumann operator associated with $- DaG + m$.
We call $\Lambda_H$ the 
{\bf Dirichlet-to-Neumann operator in~$H$ associated with $- DaG + m$}.
\end{definition}

Despite the abundance of choice of the space $H$, see Example~\ref{xbdgd205}, 
the operator $-\Lambda_H$ is always a semigroup generator.

\begin{thm} \label{tbdgd302}
Let $a \in \cl(H_1)$ and $m \in \cl(H_0)$ be coercive.
Then the Dirichlet-to-Neumann operator $\Lambda_H$ associated with $- DaG + m$
is m-sectorial.
In particular, if both $a$ and $m$ are symmetric, then 
$\Lambda_H$ is self-adjoint.
\end{thm}

The proof of this theorem is based on form methods and the next theorem.

\begin{theorem} \label{t:gen}
Let $\widetilde H,V$ be Hilbert spaces and let $j \in \cl(V,\widetilde H)$ with dense range.
Let $\gotb \colon V\times V\to \C$ be a continuous coercive sesquilinear form,
that is there exists a $\mu > 0$ such that 
$\RRe \gotb(v)\geq \mu \|v\|_V^2$ for all $v\in V$. 
Define the operator $A$ in $\widetilde H$ as follows.
Let $x,f \in \widetilde H$. 
Then $x \in \dom(A)$ and $Ax = f$ if there exists a $u \in V$ such that 
$j(u) = x$ and 
$\gotb(u,v) = (f,j(v))_{\widetilde H}$ for all $v \in V$.
Then $A$ is well-defined and m-sectorial.
If, in addition, $\gotb$ is symmetric, then $A$ is self-adjoint.
\end{theorem}
\begin{proof}
See \cite[Theorem 2.1]{AE2}.
\end{proof}

In the situation of Theorem~\ref{t:gen} we call $A$ the 
{\bf operator associated with} $(\gotb,j)$.

\medskip

Theorem~\ref{tbdgd302} is an immediate consequence of 
Theorem~\ref{t:gen} and the next proposition.

\begin{prop} \label{pbdgd303}
Let $a \in \cl(H_1)$ and $m \in \cl(H_0)$ be coercive.
Define the sesquilinear form $\gotb \colon \dom(G) \times \dom(G) \to \Ci$ by
\[
\gotb(u,v)
= (aGu,Gv)_{H_1} + (mu,v)_{H_0}
.  \]
Then $\gotb$ is coercive and continuous.
Further define $j \colon \dom(G) \to H$ by 
$j = \kappa \circ \pi_{\BD(G)}$.
Then the Dirichlet-to-Neumann operator $\Lambda_H$ associated with $- DaG + m$
is equal to the operator associated with $(\gotb,j)$.
\end{prop}
\begin{proof}
The form $\gotb$ is coercive since both $a$ and $m$ are coercive.
Obviously $\gotb$ is continuous.
Let $A$ be the operator associated with $(\gotb,j)$.
It remains to prove that $A=\Lambda_H$.

`$\Lambda_H\subset A$'.
Let $\varphi\in \dom(\Lambda_H)$ and set $\psi=\Lambda_H\varphi$.
Then there exists a $u_0\in \BD(G)$ with $\kappa(u_0)=\varphi$ and $\DtN u_0 = (G \circ \kappa^*) \psi$.
By definition there exists a $u \in \dom(D a G)$ such that 
$mu - DaGu = 0$, $u - u_0 \in \dom(\interior{G})$ and 
$\Lambda u_0 = \pi_{\BD(D)} (a G u)$.
Then $(G \circ \kappa^*) \psi = \pi_{\BD(D)} (a G u)$ and 
$j(u) = \kappa \pi_{\BD(G)} u = \kappa(u_0) = \varphi$.

Next if $v\in \dom(\interior{G})$, then
 \begin{align*}
\gotb(u,v)
&=(aGu,\interior{G} v)_{H_1} + (mu,v)_{H_0}
    \\
& =-(DaGu,v)_{H_0} + (DaGu,v)_{H_0}
=0
=(\psi,0)_H
=(\psi,j(v))_H.
 \end{align*}	
If $v \in \BD(G)$, then Lemma~\ref{l:Gdot} gives
 \begin{eqnarray*}
(\psi,j(v))_H 
& = & (\kappa^*\psi,v)_{\BD(G)}  
= (G\kappa^*\psi, Gv)_{\BD(D)} 
= (\pi_{\BD(D)} (a G u), Gv)_{\BD(D)}  \\
& = & (a G u, Gv)_{\dom(D)} 
= (a G u, Gv)_{H_1} + (D a G u, DGv)_{H_0}  \\
& = & (a G u, Gv)_{H_1} + (m u, v)_{H_0}
= \gotb(u,v)
.  
\end{eqnarray*}
Since $\dom(G) = \BD(G) \oplus \dom(\interior{G})$ it follows that 
$\gotb(u,v) = (\psi,j(v))_H$ for all $v \in \dom(G)$.
So $\varphi \in \dom(A)$ and $A \varphi = \psi$.

`$A \subset \Lambda_H$'.
Let $\varphi\in \dom(A)$ and write $\psi=A\varphi$.
Then there exists a $u\in \dom(G)$ such that $j(u)=\varphi$ and
\begin{equation} \label{eq:psa}
  (aGu,Gv)_{H_1} +  (mu,v)_{H_0}
= \gotb(u,v) 
= (\psi,j(v))_H
 \end{equation}
for all $v \in \dom(G)$.
If $v \in \dom(\interior{G})$, then
\[
(aGu,\interior{G}v)_{H_1} + (mu,v)_{H_0}
= (\psi,j(v))_H=0
.  \]
So $a G u \in \dom((\interior{G})^*) = \dom(D)$ and 
$D a G u = - (\interior{G})^* a G u = m u$.
Moreover, 
\begin{equation}
\Lambda \pi_{\BD(G)} u = \pi_{\BD(D)}(a G u)
\label{epbdgd303;1}
\end{equation}
by the definition of $\Lambda$.
Note that $\kappa(\pi_{\BD(G)} u) = j(u) = \varphi$.

Now let $v \in \BD(G)$. 
Then (\ref{eq:psa}) gives
 \begin{align*}
   (\kappa^*\psi,v)_{\BD(G)}
& = (\psi,\kappa(v))_H
   \\ & = (aGu,Gv)_{H_1} + (mu,v)_{H_0}
   \\ & = (a G u,Gv)_{H_1} + (Da G u, DGv)_{H_0}
   \\ & = (a G u,Gv)_{\dom(D)}
   \\ & = (\pi_{\BD(D)}(a G u),Gv)_{\BD(D)}
   \\ & = (D\pi_{\BD(G)}(a G u),v)_{\BD(G)},
 \end{align*}
where we used Lemma~\ref{l:Gdot} in the last step.
So, $\kappa^*\psi=D\pi_{\BD(D)}(a G u)$.
Hence 
\[
(G \circ \kappa^*)(\psi) 
= \pi_{\BD(D)}(a G u) 
= \Lambda \pi_{\BD(G)} u
\]
by Lemma~\ref{l:Gdot} and (\ref{epbdgd303;1}).
Therefore $\varphi \in \dom(\Lambda_H)$ and $\Lambda_H \varphi = \psi$.
\end{proof}

We next show that the operator $\Lambda_H$ is invertible and 
determine its inverse.

\begin{prop}\label{p:dtinv} 
The operator $\Lambda_H$ is invertible and 
\[
\Lambda_H^{-1}\psi 
= \kappa\begin{pmatrix} \pi_{\BD(G)} & 0 \end{pmatrix} 
      \begin{pmatrix} m & -\interior{D} \\ -G & a^{-1} \end{pmatrix}^{-1}
      \begin{pmatrix} 1\\ -a^{-1}G \end{pmatrix} \kappa^*\psi           
\]
for all $\psi \in H$.
\end{prop}
\begin{proof}
Since the form $\gotb$ in Proposition~\ref{pbdgd303} is coercive, it 
follows that the operator $\Lambda_H$ is invertible.
Let $\varphi \in \dom(\Lambda_H)$ and write $\psi = \Lambda_H \varphi$.
Then there exists a $u_0\in \BD(G)$ such that $\kappa(u_0)=\varphi$ and $\DtN u_0= G\kappa^*\psi$.
By Theorem~\ref{t:DtNwd} we obtain that 
  \begin{align*}
     u_0 = \DtN^{-1} G\kappa^*\psi
      & = \begin{pmatrix} \pi_{\BD(G)} & 0 \end{pmatrix} 
            \begin{pmatrix} m & -\interior{D} \\ -G & a^{-1}\end{pmatrix}^{-1}
            \begin{pmatrix} D\\ -a^{-1}\end{pmatrix}G\kappa^*\psi
     \\ & = \begin{pmatrix} \pi_{\BD(G)} & 0 \end{pmatrix} 
            \begin{pmatrix} m & -\interior{D} \\ -G & a^{-1}\end{pmatrix}^{-1}
            \begin{pmatrix} I \\ -a^{-1}G\end{pmatrix}\kappa^*\psi,
  \end{align*}
 where we used Lemma~\ref{l:Gdot} in the last step.
Next apply $\kappa$ to both sides. 
Since the inverse matrix maps $H_0 \times H_1$ into 
$\dom(G) \times \dom(D)$ by Lemma~\ref{l:1storder}\ref{l:1storder-2},
the proposition follows.
\end{proof}

\section{Resolvent convergence} \label{Sbdgd4}

In this section we consider a sequence of Dirichlet-to-Neumann operators 
and show resolvent convergence.

Throughout this section we adopt the notation and assumptions 
as in the beginning of Section~\ref{Sbdgd2}.
Let $H$ be a Hilbert space and $\kappa \in \cl(\BD(G),H)$
injective with dense range.
Further, we let $m_n,m \in \cl(H_0)$ and $a_n,a \in \cl(H_1)$
for all $n \in \Ni$.
Let $\mu > 0$ and assume that $\RRe m_n, \RRe m \geq \mu I_{H_0}$ and 
$\RRe a_n, \RRe a \geq \mu I_{H_1}$ for all $n \in \Ni$.
Moreover, assume that $\sup_n \|a_n\|_{\cl(H_1)} < \infty$.
Let $\Lambda, \Lambda_1,\Lambda_2,\ldots$ be the Dirichlet-to-Neumann operators
from $\BD(G)$ into $\BD(D)$ 
associated with $- DaG + m, - Da_1G + m_1, - Da_2G + m_2, \ldots$
as in Definition~\ref{dbdgd2-4.5}.
Similarly, let 
$\Lambda_H, \Lambda^{(1)}_H,\Lambda^{(2)}_H,\ldots$ be the Dirichlet-to-Neumann operators
in $H$ as in Definition~\ref{dbdgd301}.

Throughout this section we suppose in addition that 
the inclusion $\dom(G)\hookrightarrow H_0$ is compact.

The compactness assumption is valid in our model case, Example~\ref{ex:classical}, 
if $\Omega$ has a continuous boundary or, equivalently, 
if $\Omega$ has the segment property.

We state two well-known consequences of the compactness assumption.

\begin{lemma} \label{lbdgd401}
\mbox{}
\begin{tabel} 
\item \label{lbdgd401-1}
There exists a $c > 0$ such that 
$\|u\|_{H_0}\leq c\|Gu\|_{H_1}$
for all $u\in \dom(G)\cap \kar(G)^{\bot_{H_0}}$.
\item \label{lbdgd401-2}
The space $\rge(G)$ is closed in $H_1$.
\end{tabel}
\end{lemma}
\begin{proof}
`\ref{lbdgd401-1}'.
Suppose not.
Then there exists a sequence $(u_n)_{n \in \Ni}$ in 
$\dom(G)\cap \kar(G)^{\bot_{H_0}}$ such that $\|u_n\|_{H_0}=1$ and 
 \begin{equation}
   \|u_n\|_{H_0}\geq n\|Gu_n\|_{H_1} 
\label{elbdgd401;1}
\end{equation}
for all $n \in \Ni$.
 Then $(u_n)_{n \in \Ni}$ is bounded in $\dom(G)$.
We may assume without loss of generality that there exists a $u \in \dom(G)$ such that 
$\lim u_n = u$ weakly in $\dom(G)$.
Since the inclusion $\dom(G) \subset H_0$ is compact 
we obtain that $\lim u_n = u$ in $H_0$.
Then $u \in \ker(G)^{\perp_{H_0}}$ since $\ker(G)^{\perp_{H_0}}$ is closed in $H_0$.
Moreover, $\|u\|_{H_0} = 1$ and in particular $u \neq 0$.
Alternatively, (\ref{elbdgd401;1}) implies that
$\|Gu\|_{H_1}\leq \liminf_{n\to\infty}\|Gu_n\|_{H_1}=0$. 
So $u\in \kar(G)$.
Hence $u \in \ker(G) \cap \ker(G)^{\perp_{H_0}} = \{ 0 \} $
and $u = 0$.
This is a contradiction.

`\ref{lbdgd401-2}'.
This is a consequence of Statement~\ref{lbdgd401-1} and the closedness of $G$.
\end{proof}

We provide $\rge(G)$ with the induced norm of $H_1$.
Throughout the remainder of this section we denote by
$\iota \colon \rge(G)\hookrightarrow H_1$ the embedding map.
Note that $\iota^*$ is the orthogonal projection from $H_1$ onto $\ran(G)$.
The main result of this section is the following theorem.

\begin{thm} \label{tbdgd403}
Suppose that $\lim m_n = m$ in the weak operator topology on $\cl(H_0)$ 
and $\lim_{n \to \infty} (\iota^* a_n\iota)^{-1} = (\iota^* a\iota)^{-1}$
in the weak operator topology on $\cl(\ran(G))$.
Then 
\[
\lim (\Lambda_H^{(n)})^{-1} = \Lambda_H^{-1}
\]
in the weak operator topology on $\cl(H)$.
Moreover, if in addition the map $\kappa$ is compact, then the convergence is 
uniform in $\cl(H)$.
\end{thm}

If the $m_n$ are multiplication operators, then convergence in the 
weak operator topology can be rephrased.

\begin{example} \label{xbdgd403.4}
If $\Omega \subset \Ri^d$ is open, $H_0 = L_2(\Omega)$, 
$V_n,V \in L_\infty(\Omega)$ and $m_n,m$ are the multiplication 
operators associated with $V_n$ and $V$ for all $n \in \Ni$, then $\lim m_n = m$
in the weak operator topology on $\cl(H_0)$ if and only if
$\lim V_n = V$ in the weak$^*$-topology on $L_\infty(\Omega)$.
\end{example}

For the proof of Theorem~\ref{tbdgd403} we need some preliminary results.
The first one contains an identity for $\DtN$ involving $\ran(G)$.

\begin{lemma} \label{l:refntd} 
\mbox{}
\begin{tabel}
\item  \label{l:refntd-0.5}
Let $q \in H_1$.
Then $q \in  \dom(\interior{D})$ if and only if $\iota^* q \in \dom(\interior{D})$.
In that case $\interior{D} q = \interior{D} \iota^* q$.
\item  \label{l:refntd-1} 
The operator $\interior{D} \iota \colon \ran(G) \cap \dom(\interior{D}) \to H_0$
is a closed and densely defined operator in $\ran(G)$.
Moreover, $(\interior{D} \iota)^* = - \iota^* G$.
\item  \label{l:refntd-1.2} 
The operator $\interior{D} \iota$ is injective.
\item  \label{l:refntd-1.5}
The inclusion $\dom( \interior{D} \iota ) \subset H_1$ is compact.
\item  \label{l:refntd-2} 
The operator 
$\begin{pmatrix} m & -\interior{D}\iota \\ -\iota^*G & (\iota^*a\iota)^{-1}\end{pmatrix}
\colon \dom(G) \times \big( \ran(G) \cap \dom(\interior{D}) \big) \to H_0 \times \ran(G)$ 
is invertible.
\item  \label{l:refntd-3} 
The operator 
$\begin{pmatrix} m & -\interior{D}\iota \\ -\iota^*G & (\iota^*a\iota)^{-1}\end{pmatrix}^{-1}$
is bounded from $H_0 \times \ran(G)$ into $\dom(G) \times \dom(\interior{D})$.
\item  \label{l:refntd-4} 
If $q_0\in \BD(D)$, then
\[
\DtN^{-1} q_0 
= \begin{pmatrix} \pi_{\BD(G)} & 0 \end{pmatrix} 
   \begin{pmatrix} m & -\interior{D}\iota \\ -\iota^*G & (\iota^*a\iota)^{-1}\end{pmatrix}^{-1}
   \begin{pmatrix} D\\ -(\iota^*a\iota)^{-1}\iota^*\end{pmatrix}q_0.
\]
\end{tabel}
\end{lemma}
\begin{proof}
`\ref{l:refntd-0.5}'.
First $q - \iota^* q 
\in (\ran(G))^{\perp_{H_1}} 
= \ker(G^*) 
= \ker(\interior{D}) 
\subset \dom(\interior{D})$.
This shows the equivalence.
Since $\interior{D}(q - \iota^* q) = 0$, the last statement follows.

`\ref{l:refntd-1}'. 
Let $q \in \ran(G)$.
Since $\dom(\interior{D})$ is dense in $H_1$
there exists a sequence $(q_n)_{n \in \Ni}$ in $\dom(\interior{D})$ such that 
$\lim q_n = q$ in $H_1$.
Then $\iota^* q_n \in \ran(G) \cap \dom(\interior{D})$ for all $n \in \Ni$ by 
Statement~\ref{l:refntd-0.5} and $\lim \iota^* q_n = \iota^* q = q$ in $H_1$.
So $\ran(G) \cap \dom(\interior{D})$ is dense in $\ran(G)$.

Because $\ran(G)$ is closed in $H_1$ and $\interior{D}$ is a closed operator
one deduces easily that the operator $\interior{D} \iota$ is closed.
It remains to show that $(\interior{D} \iota)^* = - \iota^* G$.

Let $u \in \dom((\interior{D} \iota)^*)$.
Write $q = (\interior{D} \iota)^* u$.
Note that $q \in \ran(G)$.
Let $q' \in \dom(\interior{D})$.
Then Statement~\ref{l:refntd-0.5} implies that
\[
(u, \interior{D} q')_{H_0}
= (u, \interior{D} \iota^* q')_{H_0}
= (u, (\interior{D} \iota) \iota^* q')_{H_0}
= ( (\interior{D} \iota)^* u, \iota^* q')_{\ran(G)}
= (q, \iota^* q')_{\ran(G)}
= (q, q')_{H_1}
.  \]
So $u \in \dom((\interior{D})^*) = \dom(G)$
and $G u = - (\interior{D})^* u = - q$.
Therefore, $- \iota^* G u = q = (\interior{D} \iota)^* u$.
This implies that $(\interior{D} \iota)^* \subset - \iota^* G$.
The converse inclusion is easier and is left to the reader.

`\ref{l:refntd-1.2}'.
Let $q \in \ran(G) \cap \dom(\interior{D})$ and suppose that 
$\interior{D} \iota q = 0$.
There exists a $u \in \dom(G) \cap (\ker G)^{\perp_{H_0}}$ 
such that $q = Gu$.
Then 
$\|G u\|_{H_1}^2 
= - (q, (\interior{D})^*u)_{H_1}
= - (\interior{D} \iota q, u)_{H_0}
= 0$.
So $u \in \ker G$ and $u = 0$.

`\ref{l:refntd-1.5}'.
Let $q,q_1,q_2,\ldots \in \dom(\interior{D} \iota)$ and suppose that 
$\lim q_n = q$ weakly in $\dom(\interior{D} \iota)$.
For all $n \in \Ni$ there exists a unique $u_n \in \dom(G) \cap \ker(G)^{\perp_{H_0}}$ 
such that $q_n = G u_n$.
Since $\lim q_n = q$ weakly in $H_1$, the sequence $(q_n)_{n \in \Ni}$ is 
bounded in $H_1$.
Hence the sequence $(u_n)_{n \in \Ni}$ is bounded in $H_0$ by 
Lemma~\ref{lbdgd401}\ref{lbdgd401-1}.
Passing to a subsequence if necessary, there exists a $u \in H_0$ 
such that $\lim u_n = u$ weakly in $H_0$.
Since $G$ is a weakly closed operator, one deduces that $u \in \dom(G)$ and $G u = q$.
Then $\lim u_n = u$ weakly in $\dom(G)$, so $\lim u_n = u$ strongly in $H_0$
by the compactness assumption.
Note that $G^* = - \interior{D}$.
So 
\[
\lim_{n \to \infty} \|q_n\|_{H_1}^2
= \lim_{n \to \infty} (q_n, G u_n)_{H_1}
= \lim_{n \to \infty} (- \interior{D} q_n, u_n)_{H_0}
= (- \interior{D} q, u)_{H_0}
= (q, G u)_{H_0}
= \|q\|_{H_1}^2
.  \]
Hence $\lim q_n = q$ in $H_1$.

`\ref{l:refntd-2}' and `\ref{l:refntd-3}'.
This is as in the proof of Lemma~\ref{l:1storder}\ref{l:1storder-1} and~\ref{l:1storder-2}.

`\ref{l:refntd-4}'.
Let $q_0 \in \BD(D)$.
By Proposition~\ref{pbdgd204} there exists a unique $u \in \dom(DaG)$ 
such that $mu - D a G u = 0$ and $a G u - q_0 \in \dom(\interior{D})$.
Then $\Lambda^{-1} q_0 = \pi_{BD(G)} u$.
Write $q = a G u$.
Then $q - q_0 \in \dom(\interior{D})$, so 
$\interior{D} (q - q_0) 
   = \interior{D} \iota^* (q - q_0) 
   = (\interior{D} \iota) \iota^* (q - q_0)$
by Statement~\ref{l:refntd-0.5}.
Therefore 
\begin{equation}
D q_0 
= mu - \interior{D} (q - q_0)
= mu - (\interior{D} \iota) \iota^* (q - q_0)
.
\label{el:refntd;1}
\end{equation}
Also $\iota^* q = \iota^* a G u = (\iota^* a \iota) \iota^* G u$.
Hence $(\iota^* a \iota)^{-1} \iota^* q = \iota^* G u$
and 
$- \iota^* G u + (\iota^* a \iota)^{-1} \iota^* (q - q_0)
= - (\iota^* a \iota)^{-1} \iota^* q_0$.
Together with (\ref{el:refntd;1}) this gives
\[
\begin{pmatrix} m & -\interior{D}\iota \\ -\iota^*G & (\iota^*a\iota)^{-1}\end{pmatrix}
\begin{pmatrix} u \\ \iota^* (q - q_0) \end{pmatrix}
= \begin{pmatrix} D\\ -(\iota^*a\iota)^{-1}\iota^*\end{pmatrix} q_0
.
\]
Finally use Statement~\ref{l:refntd-2}.
\end{proof}

Next we need a sequential version of Lemma~\ref{l:wp}.

\begin{lemma}\label{l:con} 
Let $\widetilde H$ be a Hilbert space, $M\in \cl(\widetilde H)$ and $A$ a skew-adjoint operator in $\widetilde H$.
Further let $(M_n)_{n \in \Ni}$ be a sequence in $\cl(\widetilde H)$
and suppose that $\lim M_n = M$ in the weak operator topology on $\cl(\widetilde H)$.
Assume that the inclusion $\dom(A)\subset \widetilde H$ is compact and that there exists 
a $\lambda>0$ such that $\RRe M_n\geq \lambda I_{\widetilde H}$ for all $n\in \Ni$. 
Let $(x_n)_{n \in \Ni}$ be a sequence in $\widetilde H$ which converges weakly to $x \in \widetilde H$.
Then $M + A$ is invertible and 
$\lim_{n \to \infty} (M_n + A)^{-1} x_n = (M + A)^{-1} x$ weakly in $\dom(A)$.
\end{lemma}
\begin{proof}
Obviously $\RRe M \geq \lambda I_{\widetilde H}$, so $M + A$ is invertible by Lemma~\ref{l:wp}.
Consider $z_n = (M_n + A)^{-1} x_n$ for all $n \in \Ni$.
Then $\|z_n\|_{\dom(A)} \leq \frac{1 + \lambda + \|M_n\|}{\lambda} \|x_n\|_{\widetilde H}$ for all 
$n \in \Ni$ by Lemma~\ref{l:wp}.
So the sequence $(z_n)_{n \in \Ni}$ is bounded in $\dom(A)$.
Passing to a subsequence, we may assume without loss of generality that there 
exists a $z \in \dom(A)$ such that $\lim z_n = z$ weakly in $\dom(A)$.
Then $\lim z_n = z$ in $\widetilde H$ by the compactness assumption.
Consequently, $\lim M_n z_n = M z$ weakly in $\widetilde H$.
Now $M_n z_n + A z_n = x_n$ for all $n \in \Ni$.
Take the limit $n \to \infty$ and notice that both sides converge weakly in $\widetilde H$.
It follows that $M z + A z = x$, so $z = (M + A)^{-1} x$.
Now the lemma follows by a standard subsequence argument.
\end{proof}

We need one more convergence result for the proof of Theorem~\ref{tbdgd403}.
This result is also of independent interest.

\begin{prop} \label{pbdgd406}
Suppose that $\lim m_n = m$ in the weak operator topology on $\cl(H_0)$ 
and $\lim (\iota^* a_n\iota)^{-1} = (\iota^* a\iota)^{-1}$
in the weak operator topology on $\cl(\ran(G))$.
Let $q,q_1,q_2,\ldots \in \BD(D)$ and assume that $\lim q_n = q$ in $\BD(D)$.
Then 
\[
\lim_{n \to \infty} \Lambda_n^{-1} q_n= \Lambda^{-1} q
\]
weakly in $\BD(G)$.
\end{prop}
\begin{proof}
Choose $\widetilde H = H_0 \times \ran(G)$ and let 
$A = \begin{pmatrix} 0 & -\interior{D}\iota \\ -\iota^*G & 0 \end{pmatrix}$
with $\dom(A) = \dom(G) \times \big( \ran(G) \cap \dom(\interior{D}) \big)$.
Then $A$ is skew-adjoint in $\widetilde H$ by Lemma~\ref{l:refntd}\ref{l:refntd-1}.
Moreover, the inclusion $\dom(A) \subset \widetilde H$ is compact by 
Lemma~\ref{l:refntd}\ref{l:refntd-1.5} and the compactness assumption.
Further let
\[
M = \begin{pmatrix} m & 0 \\ 0 & (\iota^*a\iota)^{-1}\end{pmatrix}
\quad \mbox{and} \quad
M_n = \begin{pmatrix} m_n & 0 \\ 0 & (\iota^*a_n\iota)^{-1}\end{pmatrix}
\]
for all $n \in \Ni$.
Then $\lim M_n = M$ in the weak operator topology on $\cl(\widetilde H)$.
Since 
\[
\RRe (\iota^* a_n \iota)^{-1} 
\geq \|\iota^* a_n \iota\|_{\cl(\ran(G))}^{-2} \RRe (\iota^* a_n \iota)
\geq \|a_n\|_{\cl(H_1)}^{-2} \RRe (\iota^* a_n \iota)
\]
for all $n \in \Ni$ and $\sup_n \|a_n\|_{\cl(H_1)} < \infty$,
it follows that there exists a $\lambda > 0$ such that $\RRe M_n \geq \lambda I$ 
for all $n \in \Ni$.
We use Lemma~\ref{l:refntd}\ref{l:refntd-4} for $\Lambda^{-1}$ and $\Lambda_n^{-1}$.
Obviously
\[
\lim (D q_n, -(\iota^*a_n\iota)^{-1}\iota^* q_n) = (D q, -(\iota^*a\iota)^{-1}\iota^* q)
\]
weakly in $\widetilde H$.
Hence 
\[
\lim_{n \to \infty} \begin{pmatrix} m_n & -\interior{D}\iota \\ -\iota^*G & (\iota^*a_n\iota)^{-1}\end{pmatrix}^{-1}
   \begin{pmatrix} D\\ -(\iota^*a_n\iota)^{-1}\iota^*\end{pmatrix} q_n
= \begin{pmatrix} m & -\interior{D}\iota \\ -\iota^*G & (\iota^*a\iota)^{-1}\end{pmatrix}^{-1}
   \begin{pmatrix} D\\ -(\iota^*a\iota)^{-1}\iota^*\end{pmatrix} q
\]
weakly in $\dom(A)$ by Lemma~\ref{l:con}.
Consequently $\lim \Lambda_n^{-1} q_n= \Lambda^{-1} q$
weakly in $\BD(G)$ by Lemma~\ref{l:refntd}\ref{l:refntd-4}.
\end{proof}

Now we are able to prove the main theorem of this section.

\begin{proof}[{\bf Proof of Theorem~\ref{tbdgd403}.}]
Let $\psi \in H$.
Then $\lim \Lambda_n^{-1} G \kappa^* \psi = \Lambda^{-1} G \kappa^* \psi$
weakly in $\BD(G)$ by Proposition~\ref{pbdgd406}.
Hence 
\[
\lim_{n \to \infty}  (\Lambda_H^{(n)})^{-1} \psi
= \lim_{n \to \infty} \kappa \Lambda_n^{-1} G \kappa^* \psi 
= \kappa \Lambda^{-1} G \kappa^* \psi
= \Lambda_H^{-1} \psi
\]
weakly in $H$.
This proves the first statement in Theorem~\ref{tbdgd403}.

Now suppose that $\kappa$ is compact.
Suppose $\lim (\Lambda_H^{(n)})^{-1} = \Lambda_H^{-1}$ in $\cl(H)$
is false.
Passing to a subsequence if necessary, there exist $\delta > 0$ 
and $\psi_1,\psi_2,\ldots \in H$ such that 
\begin{equation}
\|(\Lambda_H^{(n)})^{-1} \psi_n - \Lambda_H^{-1} \psi_n\|_H
> \delta \|\psi_n\|_H
\label{etbdgd403;1}
\end{equation}
for all $n \in \Ni$.
Without loss of generality we may assume that $\|\psi_n\|_H = 1$
for all $n \in \Ni$.
Passing again to a subsequence if necessary, there exists a $\psi \in H$ 
such that $\lim \psi_n = \psi$ weakly in $H$.
Then $\lim \kappa^* \psi_n = \kappa^* \psi$ in $\BD(G)$ since $\kappa$ is compact.
Therefore $\lim G \kappa^* \psi_n = G \kappa^* \psi$ in $\BD(D)$.
Hence $\lim \Lambda_n^{-1} G \kappa^* \psi_n = \Lambda^{-1} G \kappa^* \psi$
weakly in $\BD(G)$ by Proposition~\ref{pbdgd406}.
Using again that $\kappa$ is compact it follows that 
$\lim (\Lambda_H^{(n)})^{-1} \psi_n = (\Lambda_H)^{-1} \psi$ in $H$.
Similarly $\lim (\Lambda_H)^{-1} \psi_n = (\Lambda_H)^{-1} \psi$ in $H$.
So $\lim \|(\Lambda_H^{(n)})^{-1} \psi_n - \Lambda_H^{-1} \psi_n\|_H = 0$.
This contradicts (\ref{etbdgd403;1}) for large $n$.
\end{proof}

In Example~\ref{xbdgd625} and Proposition~\ref{pbdgd625} we show that 
in the setting of the classical example (Example~\ref{ex:classical})
$H$-convergence implies 
$\lim_{n \to \infty} (\iota^* a_n\iota)^{-1} = (\iota^* a\iota)^{-1}$
in the weak operator topology of $\cl(\ran(G))$.
Moreover, we compare it with a condition introduced in Section~\ref{Sbdgd6}.
In the real symmetric case we prove in Example~\ref{xbdgd630} that all are
actually equivalent.

\section{The non-coercive case} \label{Sbdgd5}

In this section, we drop the coerciveness condition on $m$.
As a result the Dirichlet-to-Neumann operator can become multi-valued,
that is, it is a graph and no longer an operator.
The Dirichlet-to-Neumann graph associated with the
Schr\"odinger operator $- \Delta + m$ has been studied in \cite{AEKS} and \cite{BeE1}.

Throughout this section we adopt the notation and assumptions as in the beginning of 
Section~\ref{Sbdgd2}.
Further we fix an element $m \in \cl(H_0)$ and a coercive $a \in \cl(H_1)$.
We emphasise that we do not require that $m$ is coercive.
The definition of the Dirichlet-to-Neumann \emph{graph}, however,
remains the same as in the single-valued case in Definition~\ref{dbdgd2-4.5}.

\begin{definition} \label{dbdgd501}
Set
\[
 \DtN 
= \{ (\pi_{\BD(G)}u,\pi_{\BD(D)} aGu)\in \BD(G)\times \BD(D) :  
    u\in \dom(DaG) \mbox{ and }mu-DaGu=0\}
.  \]
We call $\DtN$ the {\bf Dirichlet-to-Neumann graph associated with} $-DaG + m$.
\end{definition}

We briefly recall some definitions in the area of (linear) graphs.
Let $H,K$ be Hilbert spaces.
Then a {\bf graph} $A$ is a vector subspace of $H \times K$.
The {\bf domain}, {\bf multi-valued part} and {\bf inverse} of $A$ are defined by
\begin{eqnarray*}
\dom(A) & = & \{ h\in H : \mbox{there exists a } k\in K 
    \mbox{ such that } (h,k)\in A\}, \\ 
\mul(A) & = & \{ k\in K : (0,k)\in A\} \mbox{ and }  \\
A^{-1} & = & \{ (k,h) \in K \times H : (h,k) \in A \} 
.
\end{eqnarray*}
We say that $A$ is {\bf single-valued} or an {\bf operator} if $\mul(A) = \{ 0 \} $.
The next lemma is trivial.

\begin{lemma} \label{lbdgd510}
\mbox{}
\begin{tabel}
\item \label{lbdgd510-1}
$\mul(\Lambda) = \{ \pi_{\BD(D)} a G u : u \in \ker(m - D a \interior{G}) \} $.
\item \label{lbdgd510-2}
If $\ker(m - D a \interior{G}) = \{ 0 \} $, then $\Lambda$ is single-valued.
\end{tabel}
\end{lemma}

As in Proposition~\ref{pbdgd303} define the sesquilinear form 
$\gotb \colon \dom(G) \times \dom(G) \to \Ci$ by
\[
\gotb(u,v)
= (aGu,Gv)_{H_1} + (mu,v)_{H_0}
.  \]
We also need the Dirichlet-version of $\gotb$ defined by 
$\interior{\gotb} = \gotb|_{\dom(\interior{G}) \times \dom(\interior{G})}$.
Then $\gotb$ and $\interior{\gotb}$ are continuous.
Hence there exist $T \in \cl(\dom(G))$ and $\interior{T} \in \cl(\dom(\interior{G}))$
such that 
$\gotb(u,v) = (T u, v)_{\dom(G)}$ for all $u,v \in \dom(G)$ and 
$\interior{\gotb}(u,v) = (\interior{T} u, v)_{\dom(\interior{G})}$ 
for all $u,v \in \dom(\interior{G})$.
Note that $\ker(\interior{T}) = \ker( m - D a \interior{G})$, since 
$(\interior{G})^* = -D$.

With a condition on $\ran(\interior{T})$ we can characterise the domain 
of the Dirichlet-to-Neumann graph~$\Lambda$.

\begin{prop} \label{pbdgd502}
Suppose that $\ran(\interior{T})$ is closed in $\dom(\interior{G})$.
Then 
\[
\dom(\Lambda) 
= \{ u_0 \in \BD(G) : (G u_0, \pi_{\BD(D)} a^* G v)_{\BD(D)} = 0 
                      \mbox{ for all } v \in \ker(m^* - D a^* \interior{G}) \}
.  \]
\end{prop}
\begin{proof}
`$\subset$'.
Let $u_0 \in \dom(\Lambda)$.
Then there exists a $u \in \dom(G)$ such that $mu - D a G u = 0$
and $u_0 = \pi_{BD(G)} u$.
Let $v \in \dom(\interior{G})$.
Then $(m u, v)_{H_0} = (DaGu,v)_{H_0} = - (a G u, \interior{G} v)_{H_1}$
and 
\begin{eqnarray*}
(\interior{T} (u-u_0), v)_{\dom(\interior{G})}
= \interior{\gotb}(u-u_0,v)  
& = & (a G(u-u_0), \interior{G} v)_{H_1} + (m (u-u_0), v)_{H_0} \\
& = & - (a G u_0, \interior{G} v)_{H_1} - (m u_0, v)_{H_0}
.  
\end{eqnarray*}
Note that $\interior{T} (u-u_0) \in \ran(\interior{T}) 
= ( \ker((\interior{T})^*) )^{\perp_{\dom(\interior{G})}}$
since $\ran( \interior{T} )$ is closed.

Now let $v \in \ker(m^* - D a^* \interior{G}) = \ker( (\interior{T})^* )$.
Then 
\begin{eqnarray*}
0 
& = & - (\interior{T} (u-u_0) ,v)_{\dom(\interior{G})}
= (a G u_0, \interior{G} v)_{H_1} + (m u_0, v)_{H_0} \\
& = & (G u_0, a^* \interior{G} v)_{H_1} + (u_0, m^* v)_{H_0} \\
& = & (G u_0, a^* \interior{G} v)_{H_1} + (D G u_0, D a^* \interior{G} v)_{H_0} \\
& = & (G u_0, a^* \interior{G} v)_{\dom(D)}  
= (G u_0, \pi_{\BD(D)} a^* \interior{G} v)_{\BD(D)}
\end{eqnarray*}
as required.

`$\supset$'.
The proof is similar and for this inclusion it is essential 
that $\ran( \interior{T} )$ is closed.
\end{proof}

\begin{cor} \label{cbdgd503}
Suppose that $\ran(\interior{T})$ is closed in $\dom(\interior{G})$.
Then 
\[
\dom(\Lambda) 
= \{ u_0 \in \BD(G) : \big( \Phi( \pi_{\BD(D)} a^* G v ) \big)(u_0) = 0
                      \mbox{ for all } v \in \ker(m^* - D a^* \interior{G}) \}
,  \]
where $\Phi \colon \BD(D) \to \BD(G)'$ is the natural unitary map as in 
Proposition~\ref{pbdgd202}.
\end{cor}

We emphasise that boundary regularity is not needed in Corollary~\ref{cbdgd503}.

The next lemma gives an easy to verify condition which implies that 
$\interior{T}$ has closed range.

\begin{lemma} \label{lbdgd504}
If the inclusion $\tau \colon \dom(\interior{G}) \to H_0$ is compact, then 
$\interior{T}$ has closed range.
\end{lemma}
\begin{proof}
There exist $\mu,\omega > 0$ such that 
$\mu \|u\|_{\dom(\interior{G})}^2 \leq \RRe \interior{\gotb}(u) + \omega \|\tau u\|_{H_0}^2$
for all $u \in \dom(\interior{G})$.
Then 
$\mu \|u\|_{\dom(\interior{G})}^2 
\leq \RRe (\interior{T} u, u)_{\dom(\interior{G})}
       + \omega (\tau^* \tau u, u)_{\dom(\interior{G})}
= \RRe ((\interior{T} + \omega \tau^* \tau) u, u)_{\dom(\interior{G})}$
for all $u \in \dom(\interior{G})$.
So $\interior{T} + \omega \tau^* \tau$ is injective and has closed range.
Similarly $(\interior{T})^* + \omega \tau^* \tau$ is injective.
So $\interior{T} + \omega \tau^* \tau$ is invertible.
Since $\omega \tau^* \tau$ is compact, the operator $\interior{T}$ is Fredholm.
In particular, the range of $\interior{T}$ is closed.
\end{proof}

Note that the operator $\tau$ is compact in the situation of Example~\ref{ex:classical}.

\begin{exam} \label{xbdgd511}
Let $\Omega \subset \Ri^d$ be a bounded Lipschitz domain with boundary $\Gamma$.
Let $G$ and~$D$ be as in Example~\ref{ex:classical}.
If $u_0 \in \BD(G)$, $v \in H_0$ and
$\Phi \colon \BD(D) \to \BD(G)'$ is the natural unitary map as in 
Proposition~\ref{pbdgd202}, then it follows from 
Example~\ref{xbdgd202.6} and Proposition~\ref{pbdgd202.5}\ref{pbdgd202.5-2} that 
\begin{eqnarray*}
\big( \Phi( \pi_{\BD(D)} a^* G v ) \big)(u_0)
& = & \langle (\nu \pi_{\BD(D)} a^* G v), \Tr u_0 \rangle_{(\Tr H^1(\Omega))' \times \Tr H^1(\Omega)}  \\
& = & \langle (\nu a^* G v), \Tr u_0 \rangle_{(\Tr H^1(\Omega))' \times \Tr H^1(\Omega)}  \\
& = & \langle (\partial_\nu^{a^*} v), \Tr u_0 \rangle_{H^{-1/2}(\partial\Omega),H^{1/2}(\partial\Omega)}
,  
\end{eqnarray*}
where $\partial_\nu^{a^*}$ is the co-normal derivative.
So Corollary~\ref{cbdgd503} gives
\[
\dom(\Lambda) 
= \{ u_0 \in \BD(G) : \langle (\partial_\nu^{a^*} v), 
                              \Tr u_0 \rangle_{H^{-1/2}(\partial\Omega),H^{1/2}(\partial\Omega)} = 0
                      \mbox{ for all } v \in \ker(m^* - D a^* \interior{G}) \}
,   \]
in agreement with \cite{McL} Proposition~4.10.
\end{exam}

Next we turn to the Neumann-to-Dirichlet graph.

\begin{prop} \label{:domNtD} 
Assume that $\rge(T)$ is closed in $\dom(G)$.
Then
\[
\dom(\Lambda^{-1})
= \{ q_0 \in \BD(D) : (Dq_0,\pi_{\BD(G)}v)_{\BD(G)}=0 
     \mbox{ for all } v\in \kar(m^*-\interior{D}a^*G) \}
.  \]
\end{prop}

Before we prove the latter proposition, we need a lemma.

\begin{lemma} \label{l:solNtD} 
Let $q_0\in \BD(D)$.
Let $f_0\in \dom(G)$ be such that 
\[
(f_0,v)_{\dom(G)} = (Dq_0,\pi_{\BD(G)}v)_{\BD(G)}
\]
for all $v\in\dom(G)$.
Let $u\in \dom(G)$.
Then the following statements are equivalent.
\begin{tabeleq}
\item  \label{l:solNtD-1} 
$Tu=f_0$.
\item \label{l:solNtD-2} 
$u\in \dom(DaG)$, $mu-DaGu=0$ and $q_0 = \pi_{\BD(D)}aGu$.
\end{tabeleq}
\end{lemma}
\begin{proof}
`\ref{l:solNtD-1}$\Rightarrow$\ref{l:solNtD-2}'.
Let $v \in \dom(G)$.
Then 
\[
(mu,v)_{H_0}+(aGu,Gv)_{H_1}
= \gotb(u,v)
= (Tu,v)_{\dom(G)}
= (f_0,v)_{\dom(G)}
= (Dq_0,\pi_{\BD(G)}v)_{\BD(G)}
.  \]
Hence $(mu,v)_{H_0} + (aGu, \interior{G} v)_{H_1} = 0$ 
for all $v\in \dom(\interior{G})$.
So $a G u \in \dom( (\interior{G})^* ) = \dom(D)$ and 
$D a G u = - (\interior{G})^* a G u = m u$.
In particular, $u \in \dom(D a G)$.
Alternatively, if $v \in \BD(G)$, then 
\begin{eqnarray*}
(Dq_0, v)_{\BD(G)}
& = & (Dq_0,\pi_{\BD(G)}v)_{\BD(G)}
= (mu,v)_{H_0}+(aGu,Gv)_{H_1}  \\
& = & (D a G u , D G v)_{H_0}+(aGu,Gv)_{H_1}
= (a G u, G v)_{\dom(D)}  \\
& = & (\pi_{\BD(D)}aGu, G v)_{\BD(D)}
= (D \pi_{\BD(D)}aGu, v)_{\BD(G)}
\end{eqnarray*}
by Lemma~\ref{l:Gdot}.
So $q_0 = \pi_{\BD(D)}aGu$.

`\ref{l:solNtD-2}$\Rightarrow$\ref{l:solNtD-1}'.
Let $v \in \dom(\interior{G})$.
Since $(\interior{G})^* = - D$ one deduces that 
\begin{eqnarray*}
(Tu,v)_{\dom(G)}
& = & \gotb(u,v)
= (aGu,\interior{G}v)_{H_1} + (mu,v)_{H_0} \\
& = & - (DaGu, v)_{H_0} + (mu,v)_{H_0} 
= 0 
= (Dq_0,\pi_{\BD(G)}v)_{\BD(G)}
= (f_0,v)_{\dom(G)}
.
\end{eqnarray*}
Alternatively, if $v \in \BD(G)$, then 
\begin{eqnarray*}
(Tu,v)_{\dom(G)}
& = & \gotb(u,v)
= (aGu, Gv)_{H_1} + (mu,v)_{H_0} \\
& = & (aGu, Gv)_{H_1} + (D a G u, D G v)_{H_0} \\
& = & (a G u, G v)_{\dom(D)}
= (\pi_{\BD(D)} a G u, G v)_{\BD(D)}
= (q_0,Gv)_{\BD(D)}  \\
& = & (D q_0,v)_{\BD(G)}
= (f_0,v)_{\dom(G)}
.
\end{eqnarray*}
So by linearity $(Tu,v)_{\dom(G)} = (f_0,v)_{\dom(G)}$ for all $v \in \dom(G)$
and $Tu = f_0$.
\end{proof}

\begin{proof}[{\bf Proof of Proposition~\ref{:domNtD}.}]
Let $q_0 \in \BD(D)$.
Let $f_0 \in \dom(G)$ be as in Lemma~\ref{l:solNtD}.
Then it follows from Lemma~\ref{l:solNtD} that $q_0 \in \dom(\Lambda^{-1})$
if and only if $f_0 \in \ran(T)$.
But $\ran(T) = (\ker(T^*))^{\perp_{\dom(G)}}$ since $\ran(T)$ is closed in 
$\dom(G)$.
Now $\ker(T^*) = \ker(m^*-\interior{D}a^*G)$ because $G^* = - \interior{D}$.
Hence $f_0 \in \ran(T)$ if and only if 
$(Dq_0,\pi_{\BD(G)}v)_{\BD(G)}=0$ for all $v\in \kar(m^*-\interior{D}a^*G)$.
\end{proof}

As in Lemma~\ref{lbdgd504} one has the following sufficient condition
for the closedness of $\ran(T)$.

\begin{lemma} \label{lbdgd506}
If the inclusion $\dom(G) \subset H_0$ is compact,
then $\ran(T)$ is closed in $\dom(G)$.
\end{lemma}

In our model case Example~\ref{ex:classical}, the inclusion $\dom(G) \subset H_0$ is compact
if $\Omega$ has a continuous boundary.

We conclude with a variant of the Dirichlet-to-Neumann graph involving 
an intermediate space as in Section~\ref{Sbdgd3}.
Throughout the remainder of this section let $H$ be a Hilbert space and 
$\kappa \in \cl(\BD(G),H)$ injective with dense range.
Define 
\begin{eqnarray*}
\Lambda_H
= \{ (\varphi,\psi) \in H \times H 
& : & 
    \mbox{there exists a } u_0 \in \BD(G) \mbox{ such that }   \\*[0pt]
& & \kappa(u_0) = \varphi
    \mbox{ and } (u_0,G \kappa^* \psi) \in \Lambda \}
.  
\end{eqnarray*}
We call $\Lambda_H$ the {\bf Dirichlet-to-Neumann graph in $H$ associated with 
$- D a G + m$}.
It follows from Lemma~\ref{lbdgd510} that $\Lambda_H$ is single-valued if
$\ker(m - D a \interior{G}) = \{ 0 \} $.

The graph $\Lambda_H$ can be described with a form.

\begin{prop} \label{p:ADtNrel}
Define $j \colon \dom(G) \to H$ by 
$j = \kappa \circ \pi_{\BD(G)}$.
Then 
\begin{eqnarray*}
\Lambda_H 
= \{ (\varphi,\psi) \in H \times H 
& : & 
      \mbox{there exists a } u \in \dom(G) \mbox{ such that }  \\
& &          j(u) = \varphi \mbox{ and } \gotb(u,v) = (\psi, j(v))_{\dom(G)}
       \mbox{ for all } v \in \dom(G) \}
. 
\end{eqnarray*}
\end{prop}
\begin{proof}
This follows as in the proof of Proposition~\ref{pbdgd303}.
\end{proof}

\begin{cor} \label{cbdgd512}
If $\ker(- D a \interior{G} + m) = \{ 0 \} $ and the inclusion 
$\dom(G) \subset H_0$ is compact, then $\Lambda_H$ is an m-sectorial operator.
\end{cor}
\begin{proof}
Let $j = \kappa \circ \pi_{\BD(G)} \colon \dom(G) \to H$
and let $V(\gotb) = \{ u \in \dom(G) : \gotb(u,v) = 0 \mbox{ for all } v \in \ker j \} $.
Then $V(\gotb) \cap \ker j = \ker(- D a \interior{G} + m) = \{ 0 \} $.
Then the statement follows from \cite{ACSVV} Theorem~8.11 and Proposition~\ref{p:ADtNrel}.
\end{proof}

Even if the inclusion $\dom(G) \subset H_0$ is compact, then in general 
$\Lambda_H$ is not an m-sectorial graph. 
A counterexample has been given in \cite{BeE2} Example~3.7.

\section{Resolvent convergence, non-coercive case} \label{Sbdgd6}

In this section we consider resolvent convergence of a sequence of Dirichlet-to-Neumann 
operators without the coercivity condition on~$m$.
Throughout this section, we adopt the notation and assumptions as in the beginning of 
Section~\ref{Sbdgd2}.
Let $H$ be a Hilbert space and let
$\kappa \in \cl(\BD(G),H)$ be one-to-one with dense range.
Set $j = \kappa \circ \pi_{\BD(G)} \colon \dom(G) \to H$.

We need a stronger version of convergence for the leading coefficients,
which we next introduce.
Let $a,a_1,a_2,\ldots \in \cl(H_1)$ be coercive.
We say that {\bf $(a_n)_{n \in \Ni}$ converges to $a$ independent 
of the boundary conditions} if for every strictly increasing 
sequence $(n_k)_{k \in \Ni}$ in $\Ni$, all $f,f_1,f_2,\ldots \in H_0$
and all $u,u_1,u_2,\ldots \in \dom(G)$ with 
\begin{align}\label{eq:convIndepBd}
\left[ \begin{array}{l}
   \displaystyle \lim_{k \to \infty} f_k = f \mbox{ weakly in } H_0 ,  \\[5pt]
   \displaystyle \lim_{k \to \infty} u_k = u \mbox{ weakly in } \dom(G) , \mbox{ and}  \\[5pt]
   \displaystyle u_k \in \dom(D a_{n_k} G)
            \mbox{ and } 
                - D a_{n_k} G u_k = f_k
            \mbox{ for all } k \in \Ni 
        \end{array} \right.
\end{align}
it follows that $\lim\limits_{k \to \infty} a_{n_k} G u_k = a G u$ 
weakly in $H_1$.

Note that $D$ is weakly closed and 
$\lim_{k \to \infty} D(a_{n_k} G u_k) = \lim_{k \to \infty} - f_k = -f$
weakly in $H_0$.
So $a G u \in \dom(D)$ and $-D a G u = f$.
In particular $u \in \dom(D a G)$.

\begin{exam} \label{xbdgd625}
In this example we show that in the classical situation, convergence of the coefficients 
independent of the boundary conditions is equivalent to the already studied notion of 
$H$-convergence, see \cite{Tartar} and \cite{MuratTartar}. 

Let $\Omega \subset \Ri^d$ be open and bounded.
Further, let $H_0$, $H_1$, $G$ and $D$ be as in Example~\ref{ex:classical}.
We identify an element of $L_\infty(\Omega,\Ci^{d \times d})$ with an
element of $\cl(H_1)$ in the natural way.
Let $a,a_1,a_2,\ldots\in L_\infty(\Omega,\Ri^{d \times d})$.
Note that we require that the matrices are real valued, but they do not have 
to be symmetric.
Suppose that $\RRe a_n \geq \mu I$ for all $n \in \Ni$, 
$\RRe a \geq \mu I$ and $\sup_n \|a_n\|_{\cl(H_1)} < \infty$.

Recall that the sequence $(a_n)_{n\in\Ni}$ is called 
$H$-convergent to $a$, if for all 
$f\in H^{-1}(\Omega)$ and for all $n\in \Ni$ with $u_n\in H_0^1(\Omega)$ satisfying
\[
   (a_n \grad u_n ,\grad v)_{L_2(\Omega)^d}=f(v)
\]
for all $v\in H_0^1(\Omega)$, it follows that $\lim_{n\to\infty} u_n = u$ weakly in 
$H_0^1(\Omega)$ and $\lim_{n\to \infty} a_n \grad u_n = a\grad u$ weakly in  
$L_2(\Omega)^d$, where $u \in H^1_0(\Omega)$ is such that 
\[
   (a\grad u,\grad v)_{L_2(\Omega)^d}=f(v)
\]
for all $v \in H^1_0(\Omega)$.

Suppose that the sequence $(a_n)_{n\in\Ni}$ is $H$-convergent to $a$.
We show that $(a_n)_{n\in\Ni}$ converges to $a$ independent of the boundary conditions.
Let $f,f_1,f_2,\ldots \in L_2(\Omega)$,
$u,u_1,u_2,\ldots \in H^1(\Omega)$ and $(n_k)_{k\in\Ni}$ satisfy \eqref{eq:convIndepBd}.
Then every subsequence 
$(a_{n_k})_{k\in\Ni}$ is $H$-convergent to $a$ by the
discussion after Definition~6.4 in \cite{Tartar}.
So without loss of generality we may assume that $n_k=k$ for all $k\in\Ni$.
As $(u_k)_{k\in\Ni}$ converges to $u$ weakly in $H^1(\Omega)$ it also converges weakly 
in $H^1_\loc(\Omega)$.
The inclusion $H^1_0(\Omega)\subset L_2(\Omega)$ is compact since 
$\Omega$ is bounded.
Hence also the inclusion 
$L_2(\Omega)\subset (H^1_0(\Omega))'=H^{-1}(\Omega)$ is compact.
Therefore $(f_k)_{k\in\Ni}$ converges strongly to $f$ in 
$H^{-1}(\Omega)\subset H^{-1}_\loc(\Omega)$.  
Then the criteria of Lemma~10.3 in \cite{Tartar} are fulfilled and we obtain that 
$(a_kGu_k)_{k\in\Ni}$ converges weakly to $aGu$ in $L_{2,\loc}(\Omega)^d$.
Since the sequence $(a_kGu_k)_{k\in\Ni}$ in $L_2(\Omega)^d$ is bounded in 
$L_2(\Omega)^d$, there exists a $q\in L_2(\Omega)^d$ and a subsequence of 
$(a_kGu_k)_{k\in\Ni}$ that weakly converges to $q$ in $L_2(\Omega)^d$.
By uniqueness of limits in $L_{2,\loc}(\Omega)^d$, we must have that $q= aGu$.
So the subsequence converges to $aGu$ in $L_2(\Omega)^d$.
Using the standard subsequence argument we 
deduce that $(a_kGu_k)_{k\in\Ni}$ converges weakly to $aGu$ in $L_2(\Omega)^d=H_1$.

Conversely, suppose that $(a_n)_{n\in\Ni}$ converges to $a$ independent of the 
boundary conditions. 
We shall prove that the sequence $(a_n)_{n\in\Ni}$ is $H$-convergent to $a$.
Let $f\in  H^{-1}(\Omega)$.
For all $n\in \Ni$ let  $u_n, u\in H_0^1(\Omega)$ satisfy
\[
   (a_n \grad u_n,\grad v)_{L_2(\Omega)^d}=f(v)
\quad \mbox{and} \quad
   (a \grad u,\grad v)_{L_2(\Omega)^d}=f(v)
\]
for all $v\in H_0^1(\Omega)$. 
We need to show that $\lim u_n = u$ weakly in $H_0^1(\Omega)$ and 
$\lim a_n\grad u_n = a\grad u$ weakly in $L_2(\Omega)^d$.
Since $L_2(\Omega)$ is dense in $H^{-1}(\Omega)$, there exists
a sequence $(f_\ell)_{\ell \in \Ni}$ in 
$L_2(\Omega)$ such that $\lim_{\ell\to\infty} f_\ell = f$ in $H^{-1}(\Omega)$.
For all $n,\ell \in \Ni$ let $u_n^\ell,u^\ell \in H_0^1(\Omega)$ be such that
\[
(a_n \grad u_n^\ell ,\grad v)_{L_2(\Omega)^d}
= (f_\ell,v)_{L_2(\Omega)}
\quad \mbox{and} \quad
 (a \grad u^\ell , \grad v)_{L_2(\Omega)^d} = (f_\ell,v)_{L_2(\Omega)}
\]
for all $v \in H^1_0(\Omega)$.

Let $\ell \in \Ni$.
We shall show that $\lim_{n \to \infty} u_n^\ell = u^\ell$
weakly in $H_0^1(\Omega)$ and 
$\lim_{n \to \infty} a_n \grad u_n^\ell = a \grad u^\ell$
weakly in $L_2(\Omega)^d$.
Note that the sequence $(u^\ell_n)_{n \in \Ni}$ is bounded in $H^1_0(\Omega)$.
Choose $(n_k)_{k \in \Ni}$ to be 
a strictly increasing sequence of natural numbers
such that $(u_{n_k}^\ell)_{k \in \Ni}$ weakly converges in $H_0^1(\Omega)$,
say to $w^\ell \in H^1_0(\Omega)$.
Since the sequence $(a_n)_{n \in \Ni}$ converges to $a$ independent of the 
boundary conditions, 
one deduces that 
$\lim_{k \to \infty} a_{n_k} \grad u_{n_k}^\ell = a\grad w^\ell$
weakly in $L_2(\Omega)^d$ and 
\[
 (a \grad w^\ell , \grad v)_{L_2(\Omega)^d} = (f_\ell,v)_{L_2(\Omega)}
\]
for all $v\in H_0^1(\Omega)$.
Uniqueness of $u^\ell$ implies $w^\ell=u^\ell$.
Hence $\lim_{n \to \infty} u_n^\ell = u^\ell$ weakly in $H_0^1(\Omega)$
by a subsubsequence argument.
Since the sequence $(a_n)_{n \in \Ni}$ converges to $a$ independent of the 
boundary conditions, we obtain 
$ \lim_{n \to \infty} a_n \grad u_n^\ell = a\grad u^\ell$ weakly in $L_2(\Omega)^d$.

By the Dirichlet-type Poincar\'e inequality there exists a $c_0 > 0$
such that $\|v\|_{L_2(\Omega)} \leq c_0 \|\grad v\|_{L_2(\Omega)^d}$
for all $v \in H^1_0(\Omega)$.
If $n,\ell \in \Ni$, then 
\begin{eqnarray*}
\mu \|\grad(u_n^\ell - u_n)\|_{L_2(\Omega)^d}^2
& \leq & \RRe (a_n \grad (u_n^\ell - u_n),\grad (u_n^\ell - u_n))_{L_2(\Omega)^d}  \\
& = & \RRe (f_\ell - f)(u_n^\ell - u_n)  \\
& \leq & \|f_\ell - f\|_{H^{-1}(\Omega)} \|u_n^\ell - u_n\|_{H_0^1(\Omega)}  \\
& \leq & (1 + c_0) \|f_\ell - f\|_{H^{-1}(\Omega)} \|\grad(u_n^\ell - u_n)\|_{L_2(\Omega)^d}
.
\end{eqnarray*}
Therefore 
\[
\|u_n^\ell - u_n\|_{H_0^1(\Omega)}
\leq (1 + c_0) \|\grad(u_n^\ell - u_n)\|_{L_2(\Omega)^d}
\leq (1 + c_0)^2 \mu^{-1} \|f_\ell - f\|_{H^{-1}(\Omega)}
.  \]
Similarly 
$\|u^\ell - u\|_{H_0^1(\Omega)} \leq (1 + c_0)^2 \mu^{-1} \|f_\ell - f\|_{H^{-1}(\Omega)}$.
If $n,\ell \in \Ni$, then 
\[
\|a_n \grad u_n^\ell - a_n \grad u_n\|_{L_2(\Omega)^d} 
\leq c \|u_n^\ell - u_n\|_{H^1_0(\Omega)} 
\leq c (1 + c_0)^2 \mu^{-1} \|f_\ell - f\|_{H^{-1}(\Omega)}
, \]
where $c = \|a\|_{\cl(H_1)} \vee \sup_{k \in \Ni} \|a_k\|_{\cl(H_1)}$.
Similarly one deduces 
$\|a \grad u^\ell - a \grad u\|_{L_2(\Omega)^d} 
\leq c (1 + c_0)^2 \mu^{-1} \|f_\ell - f\|_{H^{-1}(\Omega)}$.
Now let $v \in H^1_0(\Omega)$.
Then 
\begin{eqnarray*}
| (u_n - u,v)_{H_0^1(\Omega)}|  
& \leq & | (u_n - u_n^\ell,v)_{H_0^1(\Omega)}|
   + | (u_n^\ell - u^\ell,v)_{H_0^1(\Omega)}| 
   +  | (u^\ell - u,v)_{H_0^1(\Omega)}| \\
& \leq & (1 + c_0)^2 \mu^{-1} \| f-f_\ell\|_{H^{-1}(\Omega)} \|v\|_{H_0^1(\Omega)}
   + |(u_n^\ell - u^\ell,v)_{H_0^1(\Omega)}|   \\*
& & \hspace*{40mm} {}
   + (1 + c_0)^2 \mu^{-1} \| f-f_\ell\|_{H^{-1}(\Omega)} \|v\|_{H_0^1(\Omega)}
\end{eqnarray*}
for all $\ell, n \in \Ni$,
which yields $\lim_{n \to \infty} u_n = u$ weakly in $H_0^1(\Omega)$.
It follows similarly that $\lim_{n \to \infty} a_n\grad u_n = a\grad u$ 
weakly in $L_2(\Omega)^d$.
Hence the sequence $(a_n)_{n\in\Ni}$ is $H$-convergent to $a$.
\end{exam}

The condition
$(a_n)_{n \in \Ni}$ converges to $a$ independent 
of the boundary conditions, which we use in this section, is 
stronger than the condition used for the convergence
in Theorem~\ref{tbdgd403}.

\begin{prop} \label{pbdgd625}
Let $a,a_1,a_2,\ldots \in \cl(H_1)$ and 
$\mu > 0$.
Suppose that $\RRe a_n \geq \mu I$ for all $n \in \Ni$ and 
$\RRe a \geq \mu I$.
Suppose that $(a_n)_{n \in \Ni}$ converges to $a$ independent 
of the boundary conditions.
Further assume that the inclusion $\dom(G) \subset H_0$ is compact.
Let $\iota \colon \rge(G) \hookrightarrow H_1$ be the embedding map.
Then $\lim_{n \to \infty} (\iota^* a_n \iota)^{-1} = (\iota^* a\iota)^{-1}$
in the weak operator topology on $\cl(\ran(G))$.
\end{prop}
\begin{proof}
Let $q \in \ran G \cap \dom \interior{D}$.
Let $n \in \Ni$.
Write $r_n = (\iota^* a_n \iota)^{-1} q$.
Then $r_n \in \ran G$ and $\|r_n\|_{H_1} \leq \mu^{-1} \|q\|_{H_1}$.
There exists a $u_n \in \dom G \cap (\ker G)^{\perp_{H_0}}$
such that $G u_n = r_n$.
Then the sequence $(u_n)_{n \in \Ni}$ is bounded in $\dom G$ by 
Lemma~\ref{lbdgd401}\ref{lbdgd401-1}.
Passing to a subsequence if necessary, there exists a $u \in \dom G$
such that $\lim u_n = u$ weakly in $\dom G$.
Let $n \in \Ni$.
Then $q = \iota^* a_n \iota r_n = \iota^* a_n G u_n$.
Since $q \in \dom \interior{D}$ it follows from Lemma~\ref{l:refntd}\ref{l:refntd-0.5}
that $a_n G u_n \in \dom \interior{D}$ and 
$\interior{D} a_n G u_n = \interior{D} \iota^* a_n G u_n = \interior{D} q$.
Because $(a_n)_{n \in \Ni}$ converges to $a$ independent 
of the boundary conditions, we obtain that 
$\lim a_n G u_n = a G u$ weakly in $H_1$.
Since the operator $\interior{D}$ is closed, we obtain that 
$a G u \in \dom \interior{D}$ and $\interior{D} a G u = \interior{D} q$.
Using again Lemma~\ref{l:refntd}\ref{l:refntd-0.5}
one deduces that $\iota^* a G u \in \dom \interior{D}$ and 
$\interior{D} \iota^* a G u = \interior{D} q$.
Hence $(\interior{D} \iota) \iota^* a \iota G u = (\interior{D} \iota) q$.
Since $\interior{D} \iota$ is injective by Lemma~\ref{l:refntd}\ref{l:refntd-1.2},
it follows that $\iota^* a \iota G u = q$.
So $G u = (\iota^* a \iota)^{-1} q$.
Then 
\[
\lim (\iota^* a_n \iota)^{-1} q
= \lim r_n 
= \lim G u_n
= G u
= (\iota^* a \iota)^{-1} q
\]
weakly in $\ran G$.

Finally, since $\sup \|(\iota^* a_n \iota)^{-1}\|_{\cl(\ran G)} < \infty$
and $\ran G \cap \dom \interior{D}$ is dense in 
$\ran G$ by Lemma~\ref{l:refntd}\ref{l:refntd-1}, one 
concludes that $\lim (\iota^* a_n \iota)^{-1} = (\iota^* a \iota)^{-1}$
in the weak operator topology on $\cl(\ran(G))$.
\end{proof}

\begin{remark} \label{rbdgd626}
The above proposition is also valid if $\iota$ is replaced by the 
embedding of a closed subspace of $\ran G$ which contains 
$\ran \interior{G}$.
This is the motivation for the terminology
$(a_n)_{n \in \Ni}$ converges to $a$ independent 
of the boundary conditions.
\end{remark}

The main theorem of this section is as follows.

\begin{thm} \label{tbdgd620}
Let $a,a_1,a_2,\ldots \in \cl(H_1)$, $m,m_1,m_2,\ldots \in \cl(H_0)$ and 
$\mu > 0$.
Suppose that $\RRe a_n \geq \mu I$ for all $n \in \Ni$, 
$\RRe a \geq \mu I$ and $\sup_n \|a_n\|_{\cl(H_1)} < \infty$.
Suppose that $(a_n)_{n \in \Ni}$ converges to $a$ independent 
of the boundary conditions and $\lim m_n = m$ in the weak operator 
topology on $\cl(H_0)$.
Assume that $\ker (m_n - D a_n \interior{G}) = \{ 0 \} $ for all $n \in \Ni$ 
and $\ker (m - D a \interior{G}) = \{ 0 \} $.
Further assume that the inclusion $\dom(G) \subset H_0$ is compact.

For all $n \in \Ni$ let $\Lambda_H^{(n)}$ and $\Lambda_H$ be the 
Dirichlet-to-Neumann operators in $H$ associated with 
$- D a_n G + m_n$ and $- D a G + m$, respectively.
Then one has the following.
\begin{tabel}
\item \label{tbdgd620-1}
The sequence $(\Lambda_H^{(n)})_{n \in \Ni}$ of 
operators  is uniformly sectorial.
\item \label{tbdgd620-2}
$\lim_{n \to \infty} (\lambda I + \Lambda_H^{(n)})^{-1} 
= (\lambda I + \Lambda_H)^{-1}$
in the weak operator topology for all large $\lambda > 0$.
\item \label{tbdgd620-3}
If $\kappa$ is compact, then 
\[
\lim_{n \to \infty} (\lambda I + \Lambda_H^{(n)})^{-1} 
= (\lambda I + \Lambda_H)^{-1} 
\]
uniformly in $\cl(H)$ for all large $\lambda > 0$.
\end{tabel}
\end{thm}

The proof requires a lot of preparation.
Adopt the notation and assumptions of Theorem~\ref{tbdgd620}.
For all $n \in \Ni$ define $\gotb_n \colon \dom(G) \times \dom(G) \to \Ci$ by
\[
\gotb_n(u,v)
= (a_n G u,G v)_{H_1} + (m_n u,v)_{H_0}
\]
and define $V(\gotb_n) = \{ u \in \dom(G) : \gotb_n(u,v) = 0 \mbox{ for all } v \in \ker j \} $.
Define similarly $\gotb$ and $V(\gotb)$.

\begin{lemma} \label{lbdgd621}
For all $\varepsilon > 0$ there exists an $\omega > 0$ such that 
\[
\|u\|_{H_0}^2
\leq \varepsilon \|u\|_{\dom(G)}^2 + \omega \|j(u)\|_H^2
\]
for all $n \in \Ni$ and $u \in V(\gotb_n)$.
\end{lemma}
\begin{proof}
Let $n \in \Ni$.
Since $\ker (m_n - D a_n \interior{G}) = \{ 0 \} $, the 
restriction $j|_{V(\gotb_n)}$ is injective.
Because also the inclusion $\dom(G) \subset H_0$ is compact, it follows that 
for all $\varepsilon > 0$ there exists an $\omega > 0$ such that 
\[
\|u\|_{H_0}^2
\leq \varepsilon \|u\|_{\dom(G)}^2 + \omega \|j(u)\|_H^2
\]
for all $u \in V(\gotb_n)$.
We next show that one can choose $\omega$ uniformly in~$n$.

Suppose the lemma is false.
Then without loss of generality and passing to a subsequence if necessary
there exist $\varepsilon > 0$ and for all $n \in \Ni$ there exists a $u_n \in V(\gotb_n)$
such that 
\[
\|u_n\|_{H_0}^2
> \varepsilon \|u_n\|_{\dom(G)}^2 + n \|j(u_n)\|_H^2
.  \]
Without loss of generality we may assume that $\|u_n\|_{H_0} = 1$ for all $n \in \Ni$.
Then $\varepsilon \|u_n\|_{\dom(G)}^2 \leq 1$ for all $n \in \Ni$, 
so the sequence $(u_n)_{n \in \Ni}$ is bounded in $\dom(G)$.
Passing to a subsequence if necessary there exists a $u \in \dom(G)$ 
such that $\lim u_n = u$ weakly in $\dom(G)$.
Since the inclusion $\dom(G) \subset H_0$ is compact it follows that 
$u = \lim u_n$ in $H_0$.
In particular $\|u\|_{H_0} = 1$ and $u \neq 0$.
Also $j(u) = \lim j(u_n) = 0$ in $H$, so $u \in \ker j = \dom(\interior{G})$.

If $n \in \Ni$, then $(a_n G u_n,\interior{G} v)_{H_1} = - (m_n u_n,v)_{H_0}$
for all $v \in \dom(\interior{G}) = \ker j$, since $u_n \in V(\gotb_n)$.
Therefore $a_n G u_n \in \dom((\interior{G})^*) = \dom(D)$ and 
$- D a_n G u_n = - m_n u_n$.
Next $\lim m_n u_n = m u$ weakly in $H_0$.
Since $(a_n)_{n \in \Ni}$ converges to $a$ independent 
of the boundary conditions one deduces that $a G u \in \dom(D)$ and 
$- D a G u = - m u$.
Then $u \in \ker(m - D a \interior{G}) = \{ 0 \} $.
So $u = 0$. 
This is a contradiction.
\end{proof}

\begin{lemma} \label{lbdgs622}
There exist $\tilde \mu,\omega > 0$ such that 
\[
\tilde \mu \|u\|_{\dom(G)}^2
\leq \RRe \gotb_n(u) + \omega \|j(u)\|_H^2
\]
for all $n \in \Ni$ and $u \in V(\gotb_n)$.
\end{lemma}
\begin{proof}
Let $\tilde \omega = \mu + \sup_n \|m_n\|_{\cl(H_0)}$.
Then 
\[
\mu \|u\|_{\dom(G)}^2
\leq \RRe (a_n G u, G u)_{H_1} + \mu \|u\|_{H_0}^2
\leq \RRe \gotb_n(u) + \tilde \omega \|u\|_{H_0}^2
\]
for all $n \in \Ni$ and $u \in \dom(G)$.

Choose $\varepsilon = \frac{\mu}{2 \tilde \omega}$
and let $\omega > 0$ be as in Lemma~\ref{lbdgd621}.
Let $n \in \Ni$ and $u \in V(\gotb_n)$.
Then 
\begin{eqnarray*}
\mu \|u\|_{\dom(G)}^2
& \leq & \RRe \gotb_n(u) + \tilde \omega \|u\|_{H_0}^2  \\
& \leq & \RRe \gotb_n(u) 
   + \tilde \omega 
     \Big( \frac{\mu}{2 \tilde \omega} \|u\|_{\dom(G)}^2 + \omega \|j(u)\|_H^2 \Big)  \\
& = & \RRe \gotb_n(u) + \frac{\mu}{2} \|u\|_{\dom(G)}^2
    + \omega \tilde \omega \|j(u)\|_H^2
.
\end{eqnarray*}
So 
\[
\frac{\mu}{2} \|u\|_{\dom(G)}^2
\leq \RRe \gotb_n(u) 
    + \omega \tilde \omega \|j(u)\|_H^2
\]
and the lemma follows.
\end{proof}

Now we are able to prove Theorem~\ref{tbdgd620}.

\begin{proof}[{\bf Proof of Theorem~\ref{tbdgd620}.}]
Let $\tilde \mu,\omega > 0$ be as in Lemma~\ref{lbdgs622}.

`\ref{tbdgd620-1}'.
Set $c = \sup_{n \in \Ni} (\|a_n\|_{\cl(H_1)} + \|m_n\|_{\cl(H_0)})$.
Let $n \in \Ni$ and $\varphi \in \dom(\Lambda_H^{(n)})$.
There exists a $u \in \dom(G)$ such that 
$j(u) = \varphi$ and $\gotb_n(u,v) = (\Lambda_H^{(n)} \varphi, j(v))_H$
for all $v \in \dom(G)$.
Then $u \in V(\gotb_n)$ and
$((\Lambda_H^{(n)} + \omega I) \varphi,\varphi)_H = \gotb_n(u) + \omega \|j(u)\|_H^2$,
so 
$\RRe ((\Lambda_H^{(n)} + \omega I) \varphi,\varphi)_H 
   \geq \tilde \mu \|u\|_{\dom(G)}^2$.
Therefore 
\[
|\IIm  ((\Lambda_H^{(n)} + \omega I) \varphi,\varphi)_H |
= |\IIm \gotb_n(u)|
\leq c \|u\|_{\dom(G)}^2
\leq \frac{c}{\tilde \mu} \RRe ((\Lambda_H^{(n)} + \omega I) \varphi,\varphi)_H 
.  \]
Hence the operators $\Lambda_H^{(n)}$ are sectorial with vertex $- \omega$
and semi-angle $\arctan \frac{c}{\tilde \mu}$,
uniformly in~$n$.

`\ref{tbdgd620-2}'.
In order not to repeat part of the proof in Statement~\ref{tbdgd620-3}
we first prove something more general.
Let $\lambda > \omega$.
Let $\psi,\psi_1,\psi_2,\ldots \in H$ and suppose that $\lim \psi_n = \psi$
weakly in $H$.
We shall prove that 
$\lim (\lambda I + \Lambda_H^{(n)})^{-1} \psi_n = (\lambda I + \Lambda_H)^{-1} \psi$
weakly in $H$.

Let $n \in \Ni$.
Set $\varphi_n = (\lambda I + \Lambda_H^{(n)})^{-1} \psi_n$.
There exists a $u_n \in V(\gotb_n)$ such that $j(u_n) = \varphi_n$
and 
\begin{equation}
\gotb_n(u_n,v) + \lambda (j(u_n), j(v))_H
= (\psi_n, j(v))_H
\label{etbdgd620;2}
\end{equation}
for all $v \in \dom(G)$.
Choose $v = u_n$.
Then Lemma~\ref{lbdgs622} gives
\[
\tilde \mu \|u_n\|_{\dom(G)}^2
\leq \RRe \gotb_n(u_n) + \lambda \|j(u_n)\|_H^2
= \RRe (\psi_n, j(u_n))_H
\leq \|\psi_n\|_H \, \|j\|_{\cl(\dom(G),H)} \, \|u_n\|_{\dom(G)}
.  \]
So $\|u_n\|_{\dom(G)} \leq \tilde \mu^{-1} \|\psi_n\|_H \, \|j\|_{\cl(\dom(G),H)}$.
Since the sequence $(\psi_n)_{n \in \Ni}$ is bounded in $H$, the 
sequence $(u_n)_{n \in \Ni}$ is bounded in $\dom(G)$.
Passing to a subsequence if necessary, there exists a $u \in \dom(G)$ 
such that $\lim u_n = u$ weakly in $\dom(G)$.
Since the inclusion $\dom(G) \subset H_0$ is compact one deduces that 
$\lim u_n = u$ in $H_0$.
Then $\lim m_n u_n = m u$ weakly in $H_0$.
Moreover, $\lim \varphi_n = \lim j(u_n) = j(u)$ weakly in $H$.
Next we show that $j(u) = (\lambda I + \Lambda_H)^{-1} \psi$.

Let $n \in \Ni$.
If $v \in \ker j = \dom(\interior{G})$, then 
$\gotb_n(u_n,v) = 0$, so 
$(a_n G u_n, \interior{G} v)_{H_1} = - (m_n u_n, v)_{H_0}$.
Hence $a_n G u_n \in \dom((\interior{G})^*) = \dom(D)$ and 
$- D a_n G u_n = - m_n u_n$.
In particular, $u_n \in \dom(D a_n G)$.
Moreover, $\lim u_n = u$ weakly in $\dom(G)$ and 
$\lim m_n u_n = m u$ weakly in $H_0$.
Since $(a_n)_{n \in \Ni}$ converges to $a$ independent 
of the boundary conditions, one deduces that 
$\lim a_n G u_n = a G u$ weakly in $H_1$.

Let $v \in \dom(G)$.
If $n \in \Ni$, then (\ref{etbdgd620;2}) gives
\[
(a_n G u_n, G v)_{H_1} + (m_n u_n, v)_{H_0} + \lambda (j(u_n), j(v))_H
= (\psi_n, j(v))_H
.  \]
Taking the limit $n \to \infty$ one establishes
\[
(a G u, G v)_{H_1} + (m u, v)_{H_0} + \lambda (j(u), j(v))_H
= (\psi, j(v))_H
.  \]
So $\gotb(u,v) + \lambda (j(u), j(v))_H = (\psi, j(v))_H$.
Therefore $j(u) \in \dom(\Lambda_H)$ and 
$(\lambda I + \Lambda_H) j(u) = \psi$.
With the usual subsequence argument we proved that 
$\lim (\lambda I + \Lambda_H^{(n)})^{-1} \psi_n = (\lambda I + \Lambda_H)^{-1} \psi$
weakly in $H$.
Now Statement~\ref{tbdgd620-2} follows by choosing $\psi_n = \psi$ for all $n \in \Ni$.

`\ref{tbdgd620-3}'.
Finally suppose that $\kappa$ is compact.
Then also $j$ is compact.
Let $\lambda > \omega$.
Suppose 
$\lim (\lambda I + \Lambda_H^{(n)})^{-1} = (\lambda I + \Lambda_H)^{-1}$ in $\cl(H)$
is false.
Passing to a subsequence if necessary, there exist $\delta > 0$ 
and $\psi_1,\psi_2,\ldots \in H$ such that 
\[
\|(\lambda I + \Lambda_H^{(n)})^{-1} \psi_n - (\lambda I + \Lambda_H)^{-1} \psi_n\|_H
> \delta \|\psi_n\|_H
\]
for all $n \in \Ni$.
Without loss of generality we may assume that $\|\psi_n\|_H = 1$
for all $n \in \Ni$.
Passing again to a subsequence if necessary, there exists a $\psi \in H$ 
such that $\lim \psi_n = \psi$ weakly in $H$.
Let $u_n \in V(\gotb_n)$ and $u \in \dom(G)$ be as in Part~\ref{tbdgd620-2} for all $n \in \Ni$.
Then $\lim u_n = u$ weakly in $\dom(G)$, so 
\[
\lim_{n \to \infty} (\lambda I + \Lambda_H^{(n)})^{-1} \psi_n
= \lim_{n \to \infty} j(u_n) 
= j(u) 
= (\lambda I + \Lambda_H)^{-1} \psi
\]
in $H$ by the compactness of~$j$.
Similarly 
$\lim_{n \to \infty} (\lambda I + \Lambda_H^{(n)})^{-1} \psi 
= (\lambda I + \Lambda_H)^{-1} \psi$ in $H$.
So 
\[
\lim_{n \to \infty} 
\|(\lambda I + \Lambda_H^{(n)})^{-1} \psi_n - (\lambda I + \Lambda_H)^{-1} \psi_n\|_H = 0
.  \]
This is a contradiction.
\end{proof}

Note that the limit Dirichlet-to-Neumann graph $\Lambda_H$ is an 
operator in Theorem~\ref{tbdgd620}.
In \cite{AEKS} Theorem~5.11 a different condition on the $a_n$ is used 
to obtain resolvent convergence for symmetric operators/graphs, 
but possibly multi-valued limit graph $\Lambda_H$.
Since we do not wish to require symmetry in Theorem~\ref{tbdgd620}
and we need that the limit graph $\Lambda_H$ is m-sectorial, 
we require conveniently that all graphs are single-valued. 
See also the discussion at the end of Section~\ref{Sbdgd5}.

Finally we compare various conditions on the $a_n$ in the classical case.

\begin{example} \label{xbdgd630}
In this example we characterise the condition
$\lim_{n \to \infty} (\iota^* a_n\iota)^{-1} = (\iota^* a\iota)^{-1}$
in the weak operator topology of $\cl(\ran(G))$ in Theorem~\ref{tbdgd403}
for the classical case of Example~\ref{ex:classical} and real symmetric 
coefficients.

Let $\Omega \subset \Ri^d$ be open, bounded and connected.
Assume that $H^1(\Omega)$ embeds compactly into $L_2(\Omega)$.
Further, let $H_0$, $H_1$, $G$ and $D$ be as in Example~\ref{ex:classical}.
We identify an element of $L_\infty(\Omega,\Ci^{d \times d})$ with an
element of $\cl(H_1)$ in the natural way.
Let $a,a_1,a_2,\ldots\in L_\infty(\Omega,\Ri^{d \times d})$.
Suppose that $a_n = a_n^* \geq \mu I$ for all $n \in \Ni$
and $\sup_n \|a_n\|_{\cl(H_1)} < \infty$.
Moreover, assume that $a = a^* \geq \mu I$.
We emphasise that the $a_n$ and $a$ are real valued and symmetric.
Then the following three conditions are equivalent.
\begin{tabel}
\item \label{xbdgd630-1}
The sequence $(a_n)_{n \in \Ni}$ is $H$-convergent to~$a$.
\item \label{xbdgd630-2}
The sequence $(a_n)_{n \in \Ni}$ converges to $a$ independent 
of the boundary conditions.
\item \label{xbdgd630-3}
$\lim_{n \to \infty} (\iota^* a_n\iota)^{-1} = (\iota^* a\iota)^{-1}$
in the weak operator topology of $\cl(\ran(G))$.
\end{tabel}
We proved the implications 
\ref{xbdgd630-1}$\Rightarrow$\ref{xbdgd630-2}$\Rightarrow$\ref{xbdgd630-3}
in Example~\ref{xbdgd625} and Proposition~\ref{pbdgd625}.
So it remains to show the implcation \ref{xbdgd630-3}$\Rightarrow$\ref{xbdgd630-1}.

Suppose that $\lim_{n \to \infty} (\iota^* a_n\iota)^{-1} = (\iota^* a \iota)^{-1}$
in the weak operator topology of $\cl(\ran(G))$.
Since $a_n \geq \mu I$ for all $n \in \Ni$ and
$\sup_n \|a_n\|_{\cl(H_1)} < \infty$, it follows from 
\cite{Tartar} Theorem~6.5 that the sequence $(a_n)_{n \in \Ni}$ has a 
subsequence $(a_{n_k})_{k \in \Ni}$ which is $H$-convergent.
Hence there exist $b \in L_\infty(\Omega,\Ri^{d \times d})$ and $\nu > 0$
such that $\RRe b \geq \nu I$ and 
the sequence $(a_{n_k})_{k \in \Ni}$ is $H$-convergent to $b$.
Then $b = b^*$ by \cite{Tartar} Lemma~10.2.
It follows from the implication \ref{xbdgd630-1}$\Rightarrow$\ref{xbdgd630-3}
that $\lim_{k \to \infty} (\iota^* a_{n_k} \iota)^{-1} = (\iota^* b \iota)^{-1}$
in the weak operator topology of $\cl(\ran(G))$.
Therefore $(\iota^* a\iota)^{-1} = (\iota^* b \iota)^{-1}$
and $(\iota^* a\iota) = (\iota^* b \iota)$.
So
\[
(a \grad u,\grad v)_{L_2(\Omega)^d} 
= (b \grad u,\grad v)_{L_2(\Omega)^d}
\] 
for all $u,v \in H^1(\Omega)$.
Write $c = a - b$.
Then $(c \grad u,\grad v)_{L_2(\Omega)^d} = 0$ for all $u,v \in H^1(\Omega)$.
We shall show that $c = 0$.
Let $\tau \in C_c^\infty(\Omega)$ and $\xi \in \Ri^d$.
For all $\lambda \geq 1$ define $u_\lambda \in C_c^\infty(\Omega)$
by $u_\lambda(x) = \tau(x) \, e^{i \lambda \xi \cdot x}$.
Then 
\[
0 = (c \grad u_\lambda,\grad u_\lambda)_{L_2(\Omega)^d}
= (c ( \grad \tau + i \lambda \tau \xi), ( \grad \tau + i \lambda \tau \xi) )_{L_2(\Omega)^d}
.  \]
Dividing by $\lambda^2$ and taking the limit $\lambda \to \infty$ gives
$\int_\Omega |\tau(x)|^2 \langle c(x) \xi, \xi \rangle_{\Ci^d} \, dx = 0$.
This implies that $\langle c(x) \xi, \xi \rangle_{\Ci^d} = 0$ for almost all $x \in \Omega$.
Since $\Ri^d$ is separable and $c = c^*$, this implies that 
$c = 0$ almost everywhere.
So $b = a$ almost everywhere.
We proved that the sequence $(a_{n_k})_{k \in \Ni}$ is $H$-convergent to $a$.

It follows similarly that every subsequence of $(a_n)_{n \in \Ni}$ 
has a subsubsequence which is $H$-convergent to $a$.
Since the topology of $H$-convergence is metrisable by the discussion 
after Definition~6.4 in \cite{Tartar} one concludes that 
the sequence $(a_n)_{n \in \Ni}$ is $H$-convergent to~$a$.
This completes the proof of the implication  \ref{xbdgd630-3}$\Rightarrow$\ref{xbdgd630-1}.

In \cite{Waurick2} Theorem~1.2 it is proved that the three equivalent conditions
are also equivalent to a version of Condition~\ref{xbdgd630-3}, 
where $\grad$ is replaced by $\interior{\grad}$ and $\iota$ by the 
embedding $\rge(\interior{\grad}) \subset H_1$; see also Remark \ref{rbdgd626}.
\end{example}

\begin{remark}  \label{rbdgd421}
In the situation of Example~\ref{xbdgd630}, we deduce that convergence of 
$(a_n)_{n \in \Ni}$ in the weak$^*$ topology of $L_\infty(\Omega,\Ri^{d \times d})$ neither 
implies nor is implied by $\lim(\iota^* a_n \iota)^{-1}=(\iota^* a\iota)^{-1}$ 
in the weak operator topology of $\cl(\rge(G))$.
One may also consult \cite[Examples 3.2 and 3.4]{Waurick3} on this.
 \end{remark}

\section{More examples} \label{Sbdgd7}

The first example is from linearized elasticity.

\begin{exam} \label{xbdgd601}
Let $\Omega \subset \Ri^d$ be open.
Set 
\[
L_{2,{\rm sym}}(\Omega)
= \{ S \in L_2(\Omega)^{d \times d} : S^T = S \mbox{ a.e.} \}
.  \]
Choose $H_0 = L_2(\Omega)^d$ and $H_1 = L_{2,{\rm sym}}(\Omega)$.
Define $\widehat G \colon C_c^\infty(\Omega)^d \to L_{2,{\rm sym}}(\Omega)$ by
\[
(\widehat G u)_{kl} 
= \frac{1}{2} \Big( \partial_k u_l + \partial_l u_k \Big)
.  \]
Further define 
$\widehat D \colon C_c^\infty(\Omega)^{d \times d} \cap L_{2,{\rm sym}}(\Omega) \to L_2(\Omega)^d$
by 
\[
(\widehat D q)_k 
= \sum_{l=1}^d \partial_l q_{kl}
.  \]
Then $\dom(\widehat G)$ is dense in $H_0$ and 
$\dom(\widehat D)$ is dense in $H_1$.
Moreover, using integration by parts one deduces that (\ref{exbdgs201;1}) is valid.
Then one can apply Example~\ref{xbdgs201}.

Korn's first inequality implies that 
$\|\partial_k u_l\|_{L_2(\Omega)} \leq \sqrt{2} \|\widehat G u\|_{H_1}$
for all $u \in C_c^\infty(\Omega)^d$ and $k,l \in \{ 1,\ldots,d \} $.
So $\dom(\interior{G}) \subset H^1_0(\Omega)$.
In particular the inclusion $\dom(\interior{G}) \subset H_0$ is compact if
$\Omega$ is bounded.

Under some regularity conditions on the boundary of $\Omega$,
Korn's second inequality states that there exists a $c > 0$ such that 
$\|\partial_k u_l\|_{L_2(\Omega)} \leq c \|u\|_{\dom(G)}$
for all $u \in \dom(G)$ and $k,l \in \{ 1,\ldots,d \} $.
For example, if $\Omega$ is bounded with a Lipschitz boundary, then 
Korn's second inequality is valid.
For an easy proof see \cite{Nitsche} Section~3.
If Korn's second inequality is valid, then $\dom(G) \subset H^1(\Omega)^d$.
Consequently, if Korn's second inequality is valid and $\Omega$ has a continuous 
boundary, then the inclusion $H^1(\Omega) \subset L_2(\Omega)$ is compact
and hence the inclusion $\dom(G) \subset H_0$ is compact.
We point out that Korn's second inequality is not a necessary condition for the inclusion 
$\dom(G) \subset H_0$ to be compact, see \cite{Weck1} Theorem~1.

In particular, suppose $\Omega$ is bounded with a Lipschitz boundary
and write $\Gamma = \partial \Omega$.
Let $\sigma \in (-\infty,\frac{1}{2}]$ and set $H = H^\sigma(\Gamma)^d$.
Then $\Tr u \in H$ for all $u \in \dom(G)$.
Moreover, $\Tr|_{\BD(G)} \colon \BD(G) \to H$ is injective and has dense range.
So one can consider as in Section~\ref{Sbdgd3} a Dirichlet-to-Neumann operator
in~$H$.
Note that $\Tr|_{\BD(G)}$ is compact if $\sigma < \frac{1}{2}$.
\end{exam}

The second example is from electro-magneto statics.

\begin{exam} \label{xbdgd603}
Let $\Omega \subset \Ri^3$ be open.
Using integration by parts one deduces that 
\[
(\curl u, v)_{L_2(\Omega)^3} 
= (u, \curl v)_{L_2(\Omega)^3} 
\]
for all $u,v \in C_c^\infty(\Omega)^3$.
Therefore let $H_0 = H_1 = L_2(\Omega)^3$
and define $\widehat G = \widehat D \colon C_c^\infty(\Omega)^3 \to L_2(\Omega)^3$
by $\widehat G u = \widehat D u = i \curl u$.
Then (\ref{exbdgs201;1}) is satisfied.
Using the construction in Example~\ref{xbdgs201} one obtains a new example.
\end{exam}

\subsection*{Acknowledgements}
The authors wish to thank the referee to raise questions which improved the paper.
The third named author 
expresses his gratitude for the wonderful atmosphere and hospitality
extended to him during a two months research visit at the University
of Auckland.
Part of this work is supported by the Marsden Fund Council from Government funding,
administered by the Royal Society of New Zealand.
Part of this work is supported by the EU Marie Curie IRSES program, Project ÔAOSÕ, No.~318910.
Part of this work was carried out with financial support of the EPSRC
grant EP/L018802/2: Mathematical foundations of metamaterials:
homogenisation, dissipation and operator theory.

\newpage

\noindent
A.F.M. ter Elst \\
University of Auckland\\
Department of Mathematics \\
Auckland 1142  \\ 
New Zealand\\
{\tt terelst@math.auckland.ac.nz}
\medskip

\noindent
G. Gordon \\
University of Auckland\\
Department of Mathematics\\
Auckland 1142\\
New Zealand\\
{\tt g.gordon@auckland.ac.nz}

\medskip

\noindent
Marcus Waurick \\
Department of Mathematics and Statistics\\
University of Strathclyde\\
Livingstone Tower\\
26 Richmond Street\\
Glasgow G1 1XH\\
UK\\
{\tt marcus.waurick@strath.ac.uk}

\end{document}